\documentclass[11pt,reqno]{amsart}
\usepackage{amssymb}
\usepackage{amsfonts}
\usepackage{mathrsfs}
\usepackage{amsmath}
\usepackage{graphicx}
\usepackage{hyperref}
\usepackage{float}
\usepackage{epstopdf}
\usepackage{color}
\usepackage{bm}
\usepackage{comment}
\usepackage{soul}
\usepackage{cite}

\def\<{\langle}
\def\>{\rangle}

\usepackage[makeroom]{cancel}
\usepackage{soul}

\allowdisplaybreaks
\newtheorem{theorem}{Theorem}[section]
\newtheorem{proposition}[theorem]{Proposition}
\newtheorem{lemma}[theorem]{Lemma}
\newtheorem{corollary}[theorem]{Corollary}
\newtheorem{remark}[theorem]{Remark}

\newtheorem{definition}[theorem]{Definition}

\def\mcD{\mathcal{D}}
\def\mcE{\mathcal{E}}
\def\mcF{\mathcal{F}}
\def\mcG{\mathcal{G}}
\def\R{\mathbb{R}}

\def\bn{{(n)}}
\def\1{\mathbf{1}}

\def\wh{\widehat} 
\def\wt{\widetilde}

\definecolor{orange}{rgb}{1,0.5,0}

\numberwithin{equation}{section}

\begin{document}
\title[Homogenization of anisotropic diffusion on pre-Sierpi\'{n}ski carpets]{Homogenization of anisotropic diffusion on pre-Sierpi\'{n}ski carpets}
 
\author{Shiping Cao}
\address{Department of Mathematics, The Chinese University of Hong Kong, Shatin, Hong Kong}
\email{spcao@math.cuhk.edu.hk}
\thanks{ The research of SC is partially  supported by the General Research Fund grant (project CUHK14304626) from the Hong Kong Research Grant Council,
and by a direct grant for research (project 4053774) from the Chinese University of Hong Kong}

    \author{Zhen-Qing Chen}
\address{Department of Mathematics, University of Washington, Seattle, WA 98195, USA}
\thanks{ The research of ZC is supported in part by a Simons Foundation fund.}
\email{zqchen@uw.edu}

\author{Hua Qiu}
\address{School of Mathematics, Nanjing University, Nanjing, 210093, P. R. China.}
\thanks{ The research of HQ is partially  supported by the National Natural Science Foundation of China (Grant No. 12471087
	and 12531004).}
\email{huaqiu@nju.edu.cn}


\date{}


\begin{abstract}
We investigate the restoration of isotropy for anisotropic diffusions on  pre-Sierpi\'nski carpets in the plane, a problem  previously studied by Barlow, Hattori, Hattori and Watanabe \cite{BHHW} in which they obtained a weak homogenization property for the ratio of effective resistances of the  anisotropic diffusions. We establish 
 the weak convergence of these anisotropic diffusions  to  Brownian motion on the Sierpi\'nski carpet. As a consequence, we give an affirmative answer to the strong homogenization conjecture raised in \cite[p.3]{BHHW}.
\end{abstract}

\maketitle

\section{Introduction}

Analysis on fractals aims to extend classical theories of partial differential equations, potential theory, and stochastic analysis to spaces with intricate, often non-smooth, geometric structures. A central and challenging problem in this field is the construction of natural diffusion processes, or Brownian motions, on such spaces. Among key models, the Sierpi\'nski carpet stands out due to its simple, highly symmetric, and recursive construction while still having the essence of the non-smoothness, making it a fundamental testbed for developing the analysis on fractals.

In their seminal work, Barlow and Bass \cite{BB, BB3} pioneered the construction of Brownian motion on the standard  Sierpi\'nski carpet $F$
 in $\R^d$ with $d=2$
 in 1989  and $d\geq 3$ in 1999. Their approach took a weak sub-sequential limit of suitably time-changed reflected Brownian motions on a sequence of approximation domains $ \{F_n; n\geq 0\}$ (see Figures \ref{Fig1} and  \ref{Fig2}). It was later shown  in \cite[Remark 5.4]{BBKT} that the  sub-sequential limit can be replaced by a full sequential limit, with the time-scaling factors given by the expected hitting times of the boundary faces not containing $\bf 0$ by the reflected Brownian motion on $F_n$ starting from $\bf 0$. Recently, it was established
  in \cite[Theorem 1.1]{CC} that the time-scaling factors can be taken to be    $3^{n(d_w-d_f)}$. Here, $d_w>2$ is the walk dimension and $d_f$ is the Hausdorff dimension of $F$.  From the perspective of Dirichlet forms, this implies that (see \cite[Theorem 1.3]{CC}) the sequence of the rescaled
   Dirichlet forms
\begin{equation}\label{e:0}
3^{n(d_w-d_f)}\int_{F_n}\nabla f\cdot \nabla g\, dxdy\quad \hbox{ for }f,g\in W^{1,2}(F_n) 
\end{equation}
 is Mosco convergent, as $n\to\infty$, to a Dirichlet form $(\mcE_B,\mcF)$ associated with a  Brownian motion on the carpet $F$,
  which is unique up to a constant time change.

\begin{figure}
    \centering
    \includegraphics[width=0.28\linewidth]{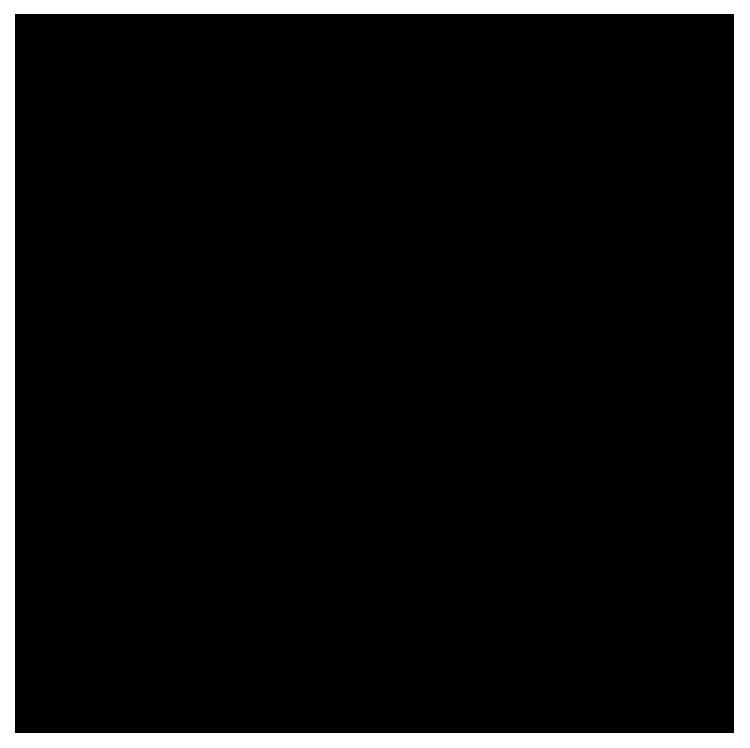}\hspace{0.4cm}
        \includegraphics[width=0.28\linewidth]{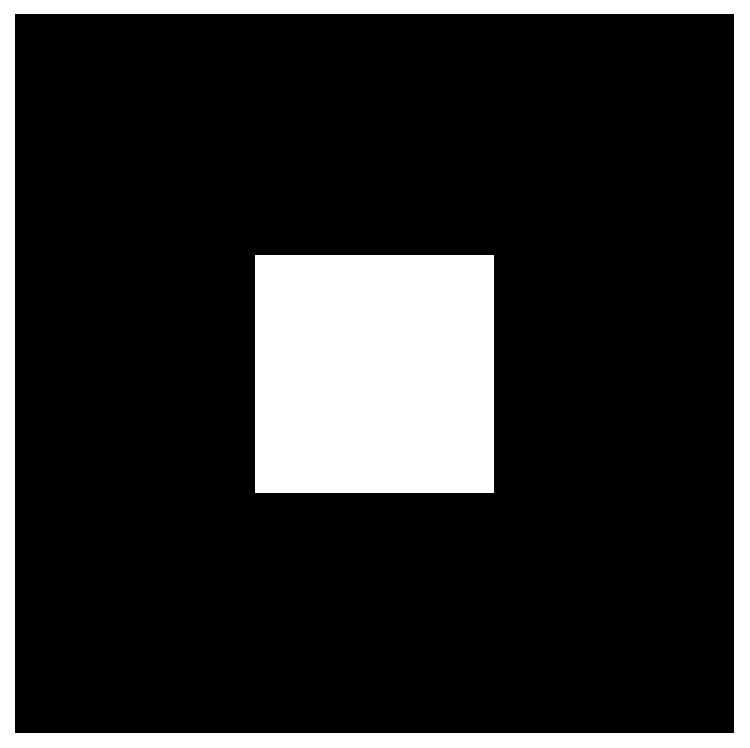}\hspace{0.4cm}
            \includegraphics[width=0.28\linewidth]{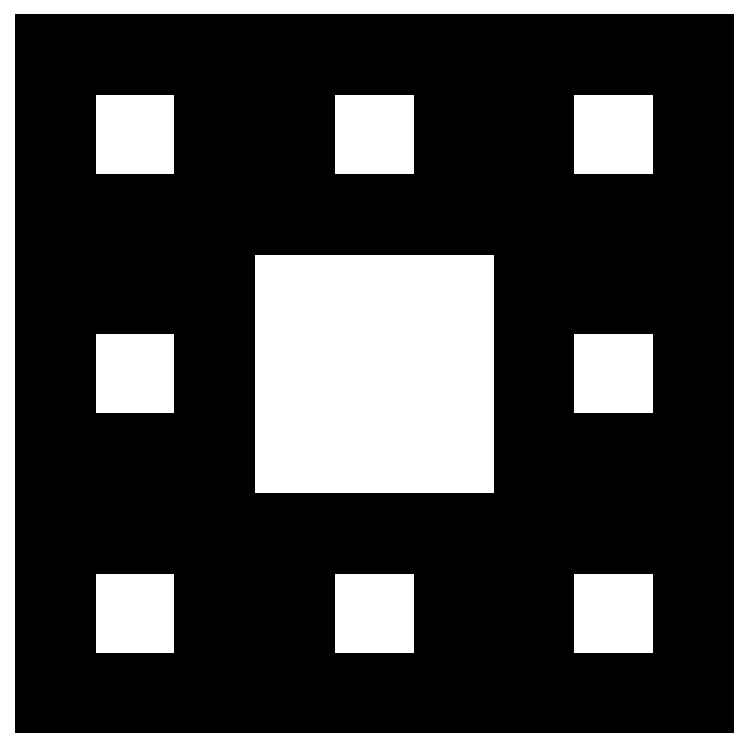}
                \caption{The pre-Sierpi\'nski carpets $F_0$, $F_1$ and $F_2$.}
                        \label{Fig1}
\end{figure}
                
\begin{figure}
        \centering
        \includegraphics[width=0.28\linewidth]{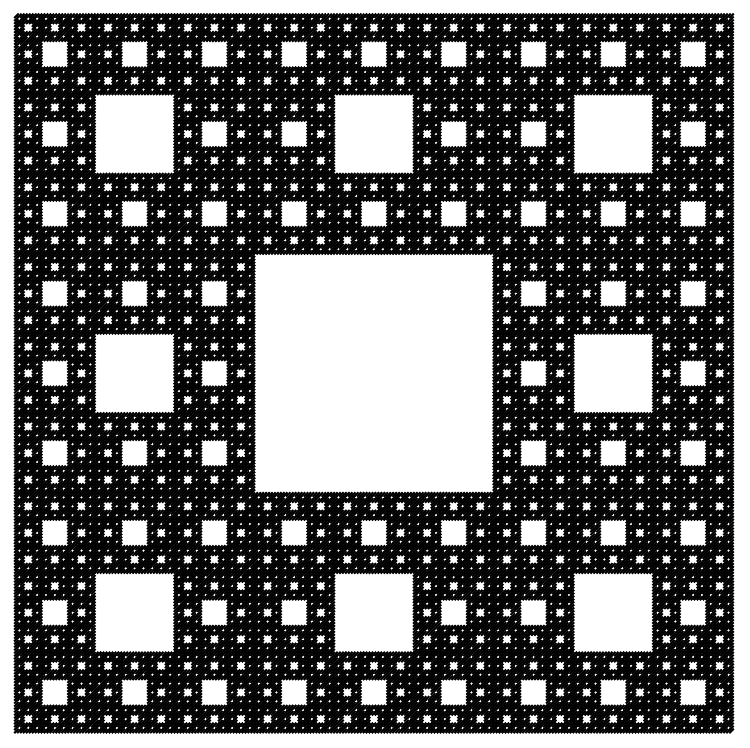}
        \caption{The Sierpi\'nski carpet $F$.}
        \label{Fig2}
    \end{figure}

    This successful construction immediately raises a fundamental question concerning uniqueness and universality: Is the Brownian motion obtained via this weak-limit approach the only natural diffusion on the carpet $F$? More precisely, do other non-standard diffusion processes, when suitably rescaled, converge to the same limit? 

In 1997,  Barlow, Hattori, Hattori and Watanabe \cite{BHHW} made  significant progress by studying a family of anisotropic diffusions on the pre-carpets  $F_n$ in $\R^2$.
 They considered the  divergence forms $\big(\mcE^{(r)}_{F_n},W^{1,2}(F_n)\big)$ defined by
\begin{equation}\label{e:1.1}
\mcE^{(r)}_{F_n}(f,g)=\int_{F_n}\left(\Big(\frac{\partial f}{\partial x}\cdot\frac{\partial g}{\partial x}\Big)(x,y)+\frac1r\Big( \frac{\partial f}{\partial y}\cdot\frac{\partial g}{\partial y}\Big)(x,y)\right)dxdy \quad\hbox{ for }f,g\in W^{1,2}(F_n),
\end{equation}
where $r\in(0,\infty)$ is an  anisotropy parameter  measuring the oscillation magnitude ratio between the horizontal and vertical directions. A key quantity in their study is the ratio $H_n(r)$ of the effective resistances $R^{(r)}_{F_n}$ in the two coordinate directions for the 
{ energy} form $\big(\mcE^{(r)}_{F_n},W^{1,2}(F_n)\big)$. Specifically,  
\begin{equation}\label{e:*}
H_n(r)=\frac{R^{(r)}_{F_n}(L_1,L_3)}{R^{(r)}_{F_n}(L_2,L_4)},
\end{equation}
where  $L_1$ and $L_3$ are the bottom and the top boundary line segment of the unit square $F_0$,
and $L_2$ and $L_4$  are the right and left boundary line segment of $F_0$.
 They established a weak homogenization theorem \cite[Theorem 1.1]{BHHW},  that is, there is constant $C>1$, independent of $r$, so that 
\[
C^{-1}\leq\liminf_{n\to\infty}H_n(r)\leq \limsup_{n\to\infty}H_n(r)\leq C.
\]
 This result hinted that, under a suitable time scaling, these anisotropic diffusions may  converge weakly to some limit  Dirichlet form on the Sierpi\'nski carpet. However, weak homogenization only asserts the  boundedness of the ratio $H_n(r)$, leaving its precise asymptotic behavior undetermined. The authors in \cite{BHHW} thus formulated a  conjecture, the \textit{Strong Homogenization Conjecture}, which states that  
\begin{equation}\label{stronghom}
\lim_{n\to\infty}H_n(r)=1\quad \hbox{ for every }r\in (0,\infty).
\end{equation}
If true, this demonstrates an intrinsic rigidity in the scaling limits of diffusions on the Sierpi\'nski carpet $F$. 
 Despite the anisotropy in the approximating processes,
the high degree of symmetry from the Sierpi\'nski carpet
 (encoded in the local and reflection symmetries) together with  
 the uniform (in $n$ and $f$) comparability of $\mcE^{(r)}_{F_n}(f,f)$ with the isotropic $\mcE^{(1)}_{F_n}(f,f)$
  collectively force any possible sub-sequential weak limit to be the isotropic standard Brownian motion. In other words, any
initial anisotropy is eventually “washed out” in the scaling limit.

The main objective of this paper is to affirm the strong  homogenization conjecture \eqref{stronghom}  in Theorem \ref{thm1}.
We further establish in Theorem \ref{thm3} that the corresponding anisotropic diffusion processes converge weakly to the Brownian motion on the
Sierpi\'{n}ski carpet constructed by Barlow and Bass \cite{BB}. In fact, we establish the weak convergence of the anisotropic diffusion processes first, 
   and then use it to prove  the strong  homogenization conjecture.

\smallskip

The core of our strategy is to address the following uniqueness question:
\textit{Are all possible sub-sequential  limit  Dirichlet forms essentially  unique?} We point out that we cannot utilize the uniqueness result from Barlow, Bass, Kumagai and Teplyaev  \cite{BBKT}
as  it is not clear a priori, due to the anisotropic nature of the Dirichlet form \eqref{e:1.1},  whether a sub-sequential limit would satisfy condition (2) of \cite[Definition 2.15]{BBKT}, 
 which requires invariance under {\it every isometry}  of $\R^2$ that maps one $n$-cell onto another. 
 Instead, we
 establish a variant uniqueness result tailored to our anisotropic diffusion model by noting that any sub-sequential limit
 Dirichlet form 
of \eqref{e:1.1} satisfies sub-Gaussian heat kernel estimates, local symmetry, and reflection symmetry (to be explicitly defined in Section \ref{sec2}). We show  in Theorem \ref{thm2} that any Dirichlet form $(\mcE,\mcF)$ with these three properties must be a constant multiple of the canonical Dirichlet form 
$(\mcE_B,\mcF)$ associated with the Brownian motion on the Sierpi\'nski carpet $F$.   
We emphasize that, in contrast to the uniqueness theorem
of \cite{BBKT}, we do not assume that the Dirichlet form $(\mcE,\mcF)$ is symmetric with respect to 
the $90$ degree counter-clockwise rotation, thereby allowing diffusions to move differently along the two coordinate directions. This uniqueness result serves as the cornerstone of our analysis. Based on it, we not only 
 obtain the strong homogenization theorem but also prove that the scaled Dirichlet form  $\big(3^{n(d_w-d_f)}\mcE^{(r)}_{F_n}, W^{1,2}(F_n)\big)$ is Mosco convergent to $(C\mcE_B,\mcF)$ as $n\to\infty$ for some constant $C>0$ depending on $r$ in Theorem \ref{thm3}.
 
 \medskip 

In this paper, we concentrate on the standard Sierpi\'nski carpet in $\R^2$.
It is natural to ask whether the strong Homogenization result and the convergence of the rescaled reflected diffusions  on pre-Sierpi\'nski carpets
extend to generalized Sierpi\'nski carpets in $\R^d$ with $d\geq 3$. While it is quite plausible, the presence of more directions  in higher dimensional spaces
makes our flow-gluing techniques developed in Section \ref{sec4.3} much more difficult to carry out. We leave this question to future research.

\medskip 

 In this paper, we use := as a way of definition. For $a, b\in \R$, $a\vee b:=\max\{a, b\}$ and $a \wedge b:=\min\{a, b\}$.
The rest of the paper is organized as follows.  

\medskip 

In Section 2, we introduce notation and the main theorems of this paper. In particular, in Section 2.1, we define the local symmetry and reflection symmetry of Dirichlet forms on the Sierpi\'nski carpet $F$. In Section 3, we introduce the restriction of energy on cells, and discuss their properties.
In Section 4, we establish upper and lower bound estimates of $\mcE/\mcE_B$ in terms of effective resistances between various pairs of subsets, where $\mcE$ is a general Dirichlet form that satisfies  local symmetry, reflection symmetry and sub-Gaussian heat kernel estimates. 
Finally, in Section 5, we use the estimates derived in Section 4 and the technique from
Barlow-Bass-Kumagai-Teplyaev \cite{BBKT} to prove the uniqueness theorem, Theorem \ref{thm2}. 
The Mosco convergence of anisotropic divergence forms (Theorem \ref{thm3}) and their strong homogenization property
 (Theorem \ref{thm1}) then follow from the uniqueness of the limiting forms, by using the techniques developed in \cite{CC}. 
   Their proofs are presented in Section 6.

\section{Main theorems}\label{sec2}

In this section, we  carefully present  the  main results of this paper.  

\medskip

We begin by defining  the Sierpi\'nski carpet (SC) and the pre-Sierpi\'nski carpets (pre-SC) in $\R^2$. Let $F_0=[0,1]^2$ be the unit square, with its four corners denoted by
\[
q_1=(0,0),\quad  q_2=(1,0),\quad  q_3=(1,1)
\quad \hbox{and} \quad q_4=(0,1).
\]
 We write 
\[
L_1=\overline{q_1q_2},\quad  L_2=\overline{q_2q_3},\quad  L_3=\overline{q_3q_4},
\quad \hbox{and} \quad L_4=\overline{q_4q_1}
\]
for the four boundary segments of $F_0$. For $i\in\{1,2,3,4\}$,  let $q_{i+4}$ be the midpoint of $L_i$, namely 
\[
q_5=(1/2,0),\quad q_6=(1,1/2),\quad  q_7=(1/2,1), \quad \hbox{and} \quad q_8=(0,1/2). 
\]
Consider the  iterated function system $\{\Psi_i\}_{1\leq i\leq 8}$ where each contraction $\Psi_i$ is given by
\[
\Psi_i(z)=\frac13z+\frac23q_i \quad \hbox{ for }i=1,2,\cdots,8,
\]
so that $q_i$ is the fixed point of $\Psi_i$.
For $n\geq 1$,  the \textit{level-$n$ pre-Sierpi\'nski carpet (pre-SC)} $F_n$ is defined inductively as 
\[
F_n=\bigcup_{i=1}^8\Psi_i(F_{n-1}).
\]
The \textit{Sierpi\'nski carpet (SC)}, denoted by $F$, is the limit of this nested sequence:
\[
F=\bigcap_{n=0}^\infty F_n.
\]
Throughout this paper,  $\mu_n$ is
 the normalized Lebesgue measure on $F_n$ satisfying $\mu_n(F_n)=1$, 
$\mu$ is the normalized $d_f$-dimensional Hausdorff measure on $F$ with $\mu(F)=1$,
 and $d_f:=\frac{\log8}{\log3}$ is
  the Hausdorff dimension of $F$. We  denote by $\partial_oF :=\bigcup_{i=1}^4L_i$
   the boundary of $F$. See Figures \ref{Fig1}, \ref{Fig2} for an illustration of the pre-SCs and the SC.\medskip

We now introduce further notation needed for the  main results.\medskip

\noindent(\textbf{Cells}) Let $W_0=\{\emptyset\}$ and $W_n=\{1,2,\cdots,8\}^n$ for $n\geq 1$ be the sets of \textit{words} of  length $n$. For a word $w=w_1w_2\cdots w_n\in W_n$ with $n\geq 1$, define
\[
\Psi_w=\Psi_{w_1}\circ\Psi_{w_2}\circ\cdots\circ\Psi_{w_n},
\]
and we use the convention $\Psi_\emptyset=\text{Id}$, where $\text{Id}$ denotes the identity map. From time to time, we denote by $\dot{w}^k=w\cdots w$  the concatenation of $k$ copies of $w$.

Each word $w\in W_n,n\geq 0$ corresponds to a \textit{level-$n$ cell} $F^w:=\Psi_w(F)$. We denote by $\partial_oF^w=\Psi_w(\partial_oF)$ the boundary of $F^w$. For a subset $A\subset W_n$, we write $F^A:=\cup_{w\in A}F^w$ for brevity.\medskip 

\noindent(\textbf{Folding maps}) 
Let $\bar{\varphi}: \mathbb{R} \to [0,1]$ be the periodic extension of the function $x \mapsto |x|$ on $[-1, 1]$; that is, $\bar{\varphi}(x) = |x|$ for $|x| \leq 1$, and $\bar{\varphi}(x + 2n) = \bar{\varphi}(x)$ for all $x \in \mathbb{R}$, $n \in \mathbb{Z}$.

Define the map $\varphi:\R^2\to F_0$ by  
\[
\varphi(x,y)=\big(\bar{\varphi}(x),\bar{\varphi}(y)\big).
\]
For each $w\in W_n$ with $n\geq 0$,  define  the \textit{folding map} $\varphi_w: F\to F^w$ as 
\[
\varphi_w(z)=3^{-n}\varphi\big(3^n(z-\Psi_w(q_1))\big)+\Psi_w(q_1). \smallskip 
\]

\noindent(\textbf{Reflection symmetries}) Let $\mathcal{G}_x$ and $\mathcal{G}_y$ denote the \textit{reflections} about the lines $x=\frac12$ and $y=\frac12$, respectively, defined by
\[
\mathcal{G}_x(x,y)=(1-x,y)\ \hbox{ and }\ \mathcal{G}_y(x,y)=(x,1-y).  
\]

\subsection{Strong homogenization of anisotropic diffusions on the pre-SCs}\label{sec21}

Let $r>0$ and $n\geq 0$. Consider the following symmetric strongly local Dirichlet from  on 
 the level-$n$ pre-Sierpi\'nski carpet $F_n$:
\[
\mcE^{(r)}_{F_n}(f,g)
=\int_{F_n}\left(\Big(\frac{\partial f}{\partial x}\cdot\frac{\partial g}{\partial x}\Big)(x,y)+\frac1r\Big(\frac{\partial f}{\partial y}\cdot\frac{\partial g}{\partial y}\Big)(x,y)\right)dxdy \quad\hbox{ for }f,g\in W^{1,2}(F_n).
\]
We define the \textit{effective resistance} between disjoint closed sets $A,B\subset F_n$ as 
\[
R^{(r)}_{F_n}(A,B)=\left(\inf\Big\{\mcE^{(r)}_{F_n}(f,f):\,f\in C(F_n)\cap W^{1,2}(F_n),\,f|_A=0,\,f|_B=1\Big\}\right)^{-1}. 
\]
In \cite[Theorem 1.1]{BHHW}, Barlow, Hattori, Hattori and Watanabe established the following weak homogenization result concerning the ratio 
\[
H_n(r):=\frac{R^{(r)}_{F_n}(L_1,L_3)}{R^{(r)}_{F_n}(L_2,L_4)}
\]
of the effective resistances in the two coordinate directions. 

\begin{proposition}[Theorem 1.1 of \cite{BHHW}]
 There exists a constant $C\in[1,\infty)$ such that 

\[
C^{-1}\leq\liminf_{n\to\infty}H_n(r)\leq \limsup_{n\to\infty}H_n(r)\leq C\quad\hbox{ for all }r\in(0,\infty). 
\]
\end{proposition}

A ``strong homogenization'' result was conjectured in \cite[Conjecture, Page 3]{BHHW}. The main goal of this paper is to affirm this conjecture,
and  to further 
show that the reflected diffusion on $F_n$ associated with  $(\mcE^{(r)}_{F_n}, W^{1,2}(F_n))$ on $L^2(F_n; \mu_n)$, after suitable scaling,
converges weakly to a Brownian motion on the Sierpi\'nski carpet $F$. The following is the first main result of this paper.

\begin{theorem}[Strong homogenization]\label{thm1}
\[
\lim_{n\to\infty} H_n(r)=1\quad\hbox{ for all }r\in(0,\infty). 
\]
\end{theorem}

\subsection{Uniqueness of a class of locally symmetric diffusions on the SC}\label{sec22} We will
  prove Theorem \ref{thm1} by investigating the uniqueness of all possible scaling limits of the forms $\big(\mcE^{(r)}_{F_n},W^{1,2}(F_n)\big)$. See Section \ref{sec23} for the definition of the Mosco convergence.

As we will see, all possible limiting forms  $(\mcE,\mcF)$ on $L^2(F;\mu)$ enjoy the following good properties.\medskip

\noindent\big(${\bf HK}(d_w)$\big) 
We say a Dirichlet form $(\mcE,\mcF)$ satisfies \textit{sub-Gaussian heat kernel estimates}, denoted by ${\bf HK}(d_w)$, if the heat semigroup associated with $(\mcE,\mcF)$ on $L^2(F;\mu)$ admits a transition density $p_t(z_1,z_2)$ satisfying
\begin{align*} 
C_1\frac{1} {t^{d_f/d_w} } \exp&\left(-C_2\Big(\frac{|z_1-z_2|^{d_w}}{t}\Big)^{\frac{1}{d_w-1}}\right)\leq p_t(z_1,z_2)\\&\leq C_3\frac{1}{t^{d_f/d_w} } 
\exp\left(-C_4\Big(\frac{|z_1-z_2|^{d_w}}{t}\Big)^{\frac{1}{d_w-1}}\right)
\end{align*} 
for all $t>0$ and $z_1,z_2\in F$, where $C_1$--$C_4$ are positive constants. 
Here  $d_w \geq 2$ is a positive parameter called the \textit{walk dimension}. 

\smallskip

It is well-known that the Brownian motion on $F$ constructed by Barlow and Bass has ${\bf HK}(d_w)$ property, see \cite{BB2}.

\begin{remark} \rm 
On a general metric measure space, if a strongly local regular Dirichlet form satisfies ${\bf HK}(d_w)$, then $d_w\geq 2$, and  $d_w$ is uniquely characterized as the critical exponent of a class of Besov spaces $W^{\beta/2,2}$, $\beta>0$. See \cite{GHL} for a proof. 
\end{remark}

\begin{remark}\label{remark24} \rm 
According to \cite[Theorem 4.2]{GHL}, if a Dirichlet form $(\mathcal{E},\mathcal{F})$ on $L^2(F;\mu)$ satisfies $\mathbf{HK}(d_w)$, then its domain $\mathcal{F}$ coincides with a Besov-type space on $(F,|\cdot|,\mu)$ and $\mathcal{E}(f,f)$ is comparable to  the corresponding Besov norm. Therefore, the notation $\mathcal{F}$ could be consistently used for such forms.
The Dirichlet form $(\mathcal{E}_B, \mathcal{F})$,  defined as the Mosco limit of the rescaled standard divergence forms in \eqref{e:0}, is associated with Brownian motion on $F$ and satisfies $\mathbf{HK}(d_w)$ (see \cite{BB2,BB3}).
\end{remark}

\begin{remark}\label{remark25}  \rm
The Sierpi\'nski carpet $F$ has Hausdorff dimension $d_f=\frac{\log8}{\log 3}$. Moreover, $3/2\geq 3^{d_w-d_f}\geq 7/6$ by \cite[Remark 5.4]{BB3}. Hence 
\[
d_w\geq \frac{\log(7/6)}{\log3}+d_f=\frac{\log(7/6)}{\log3}+\frac{\log8}{\log3}=\frac{\log(28/3)}{\log3}>2>d_f.
\]	
As a consequence, $\mcF\subset C(F)$ by ${\bf HK}(d_w)$ and the Sobolev embedding theorem (see \cite[Theorem 4.2]{HZ}). 
 
For any generalized Sierpi\'nski carpet, it is shown in \cite{Ka} that $d_w>2$.   
However, the exact value of $d_w$ remains unknown even for the classical Sierpi\'nski carpet $F$ (see \cite{BB4}). 
\end{remark}

We consider the following two types of symmetric properties of $(\mcE,\mcF)$.\medskip

\noindent(\textbf{Local symmetry})  A Dirichlet form $(\mcE,\mcF)$ on $L^2( F; \mu)$
is said to be \textit{locally symmetric} if for every $n\geq 0$ and every $f\in \mcF$, we have
\[
f\circ \varphi_w\in \mcF\quad \hbox{ for all }w\in W_n,
\]
and 
\[
\mathcal{E}(f, f)=8^{-n}\sum_{w\in W_n}\mathcal{E}\left(f\circ\varphi_w, f\circ\varphi_w\right). 
\]

\begin{remark} \rm 
The function $f\circ\varphi_w$ is constructed by ``unfolding''  copies of the restriction $f|_{F^w}$ across $F$. 
In \cite{BBKT}, the map $f\to f\circ\varphi_w$ is denoted by $U_SR_Sf$, where $R_S$ is the restriction operator and $U_S$ is the unfolding operator. 
The {\rm(\textbf{Local symmetry})} condition is exactly \cite[Definition 2.15 (1) and (3)]{BBKT}.  
\end{remark}

\smallskip

\noindent(\textbf{Reflection symmetry})   A Dirichlet form $(\mcE,\mcF)$ on $L^2(F; \mu)$
 is  said to be \textit{reflection symmetric} if for every $f\in\mcF$, we have $f\circ\mcG_x$, $f\circ\mcG_y\in\mcF$, and
\[
\mcE(f,f)=\mcE(f\circ\mcG_x,f\circ\mcG_x)=\mcE(f\circ\mcG_y,f\circ\mcG_y). 
\]
The {\rm (\textbf{Reflection symmetry})} is a special case of \cite[Definition 2.15 (2)]{BBKT}.

\smallskip

{The following is the second main result of this paper.}

\begin{theorem}\label{thm2}
There is a unique (up to a constant multiplier) symmetric strongly local regular Dirichlet form  $(\mcE,\mcF)$ on $L^2(F;\mu)$ that satisfies ${\bf HK}(d_w)$, local symmetry and reflection symmetry. 
\end{theorem}
 
 \smallskip

\begin{remark} \rm  \begin{enumerate}
\item[{\rm(a)}]  Clearly, the unique Dirichlet form corresponds precisely to the Brownian motion constructed by Barlow and Bass \cite{BB}, whose uniqueness was proven in \cite{BBKT}. 

\item[{\rm(b)}]  In contrast to \cite{BBKT}, we do not impose rotation symmetries  for $(\mcE,\mcF)$, thus allowing the form to exhibit distinct behaviors along the two coordinate directions. 
\end{enumerate}
\end{remark}

\subsection{Convergence of anisotropic diffusions on the pre-SCs}\label{sec23} Using the uniqueness theorem, Theorem \ref{thm2},  we will see that the anisotropic diffusions on the pre-SCs converge weakly to the Brownian motion on the SC. This convergence can also be interpreted in terms of  Mosco convergence of the associated Dirichlet forms combined with the tightness of diffusion processes. We now state these results more precisely. \smallskip

We begin by recalling the notions of  strong and weak convergence of $L^2$ functions on varying spaces. For strong convergence, we use the equivalent formulation given in \cite[Lemma 2.7]{CC}. 

\begin{definition}[Strong and weak convergences]\rm

\begin{enumerate}
\item[(a)]  A sequence $\{f_n\} \subset L^2(F_n;\mu_n)$ is said to \emph{ converge strongly  in $L^2$} to $f\in L^2(F;\mu)$ if there exists a sequence  $\{g_m\}\subset C(F_0)$,  such that 
\[
\lim_{m\to\infty}\big\|f-g_m|_F\big\|_{L^2(F;\mu)}=0
\] 
and 
\[
\lim_{m\to\infty}\limsup_{n\to\infty}\big\|f_n-g_m|_{F_n}\big\|_{L^2(F_n;\mu_n)}=0.
\]

\item[(b)]  A sequence $\{f_n\} {\subset} L^2(F_n;\mu_n)$ is said to \emph{converge weakly  in $L^2$} to $f\in L^2(F;\mu)$ if 
\[
\lim_{n\to\infty}\langle f_n,g_n\rangle_{L^2(F_n;\mu_n)}=\langle f,g\rangle_{L^2(F;\mu)}
\]
for every sequence $\{g_n\}$ with $g_n\in L^2(F_n; \mu_n)$ that converges strongly to $g\in L^2(F; \mu)$ as $n\to \infty$.
\end{enumerate}
\end{definition}

\begin{definition}[Mosco convergence]\label{def28} \rm 
	Let $\bar{\R}=\R\cup\{\pm\infty\}$ denote the extended real numbers. Suppose that 
	 $\wt\mcE_{F_n}:L^2(F_n; \mu_n)\to \bar{\R}$ for $n\geq 0$,  and $\mcE:L^2(F; \mu)\to \bar{\R}$ are given functionals. 
	 	We say $\wt\mcE_{F_n}$  \emph{is Mosco convergent} to $\mcE$ if and only if the following hold:
	
	\begin{enumerate}
\item[\rm (M1)]   For any sequence $\{f_n\}$ with $f_n\in L^2(F_n; \mu_n)$ converging weakly in $L^2$ to $f\in L^2(F; \mu)$ as $n\to \infty$, we have 
\[
\liminf_{n\to\infty} \wt\mcE_{F_n}(f_n)\geq\mcE(f).
\]
	
\item[\rm (M2)]  For every $f\in L^2(F; \mu)$, there exists a sequence $\{f_n\}$ with $f_n\in L^2(F_n; \mu_n)$  converging strongly in $L^2$ to $f$ as $n\to \infty$ such that 
\[
\limsup_{n\to\infty} \wt\mcE_{F_n}(f_n)\leq \mcE(f).
\]
\end{enumerate} 
\end{definition}

It is well-known that there is a one-to-one correspondence between non-negative lower-semicontinuous quadratic forms (taking values in $\bar{\mathbb{R}}$) on a Hilbert space $H$ and closed symmetric non-negative definite bilinear forms (see \cite[Section 1 (b)]{Mosco}) described as follows. Given a non-negative lower-semicontinuous quadratic form $\mcE:H  \to \bar{\R}$,  define the domain $\hbox{Dom} (\mcE)=\{f\in H:\mcE(f)< + \infty\}$ and a  symmetric bilinear form on $\hbox{Dom}(\mcE)$ via the polarization identity: 
$$\mcE(f,g):=\frac{1}{4}\Big(\mcE(f+g)-\mcE(f-g)\Big) \quad \hbox{for }  f,g\in \hbox{Dom} (\mcE).
$$
Then $\big(\mathcal{E}, \hbox{Dom}(\mcE)\big)$ is a closed symmetric non-negative definite bilinear form. 
 Conversely, given such a bilinear form $\big(\mcE,\hbox{Dom}(\mcE)\big)$, the associated quadratic form is defined by 
 $$
 \mcE(f):= \begin{cases} \mcE(f,f) \quad &\hbox{if } f\in \hbox{Dom}(\mcE), \cr
  + \infty   &\hbox{if  } f\in H\setminus \hbox{Dom}(\mcE). 
   \end{cases}
   $$
   In what follows, we keep the convention that $\mcE(f):=\mcE(f,f)$ for any symmetric  bilinear form $\big(\mcE,\hbox{Dom}(\mcE)\big)$.

Recall that $(\mcE_B,\mcF)$ is the Dirichlet form of the Brownian motion on $F$ that is the Mosco limit, as $n\to\infty$, of the rescaled standard divergence forms in \eqref{e:0} (by \cite[Theorem 1.3]{CC}).
 
 \medskip
 
 {The following is the third main result of this paper.}
 
\begin{theorem}\label{thm3}
The scaled Dirichlet form $\big(3^{n(d_w-d_f)}\mcE^{(r)}_{F_n},W^{1,2}(F_n)\big)$ is Mosco convergent to $(C\mcE_B,\mcF)$ as $n\to\infty$ for some constant $C>0$ depending on $r$. Let $(X_{n,t})_{t\geq0}$ be the Hunt process associated with $\big(3^{n(d_w-d_f)}\mcE^{(r)}_{F_n},W^{1,2}(F_n)\big)$, and let $X_t$ be the Brownian motion on $F$  associated with $(\mcE_B,\mcF)$. Then, for each $x\in F$ and $x_n\in F_n,\,n\geq0$ such
that $x_n\to x$ as $n \to \infty$, the law of $(X_{n,t})_{t\geq 0}$ starting from $x_n\in F_n$ converges weakly to the law of $(X_{Ct})_{t\geq0}$ starting from $x\in F$ in the space $C([0,\infty);F)$ equipped with the local uniform topology.
\end{theorem}

In the rest of this paper, we fix a   symmetric strongly local  regular Dirichlet form $(\mcE,\mcF)$ on $ L^2(F; \mu)$ that satisfies ${\bf HK}(d_w)$, local symmetry and reflection symmetry. Note that from Remark \ref{remark25}, we have $\mcF\subset C(F)$.

\section{Restriction of $(\mcE,\mcF)$ on cells}
For each $w\in W_n$ with $n\geq 0$, we define  $(\mcE^w,\mcF^w)$  on $L^2(F^w;\mu|_{F^w})$
 by 
\begin{align*}
\mcF^w&:=\big\{f\in L^2(F^w;\mu|_{F^w}):\,f\circ\varphi_w\in\mcF\big\},\\
\mcE^w(f,g)&:=8^{-n}\mcE(f\circ\varphi_w,g\circ\varphi_w)\quad\hbox{ for }f,g\in \mcF^w.
\end{align*} 
Recall that by \cite[Proposition 2.20]{BBKT}, $(\mcE^w,\mcF^w)$ is a strongly local, regular Dirichlet form on $L^2(F^w;\mu|_{F^w})$.  
In this short section, we summarize some basic properties of $(\mcE^w,\mcF^w)$. In particular, a suitable dilation of $(\mcE^w,\mcF^w)$ yields a strongly local, regular Dirichlet form on $L^2(F;\mu)$ that also satisfies ${\bf HK}(d_w)$, local symmetry and reflection symmetry. 

\begin{remark}\rm 
For simplicity, for a subset $A\subset W_n,n\geq 1$, we also write 
\begin{align*}
\mcF^A&=\big\{f\in L^2(F^A;\mu|_{F^A}):\,f|_{F^w}\in\mcF^w\hbox{ for }w\in A\big\},\\
\mcE^A(f,g)&=\sum_{w\in A}\mcE^w(f|_{F^w}
,g|_{F^w})\quad\hbox{ for }f,g\in \mcF^A.
\end{align*} 
Note that here we do not have $\mcF^A\subset C(F^A)$. In particular, $\mcF\subsetneq\mcF^{W_n}$.
\end{remark}

\medskip 

\noindent(\textbf{\emph{Energy measures}}). Let $(\wt{\mcE},\wt{\mcF})$ be a strongly local,  symmetric regular Dirichlet form on a locally compact, separable metric measure space $(\mathcal{X},d,m)$. For $f\in L^\infty(\mathcal{X};m)\cap \wt{\mcF}$, we define the \textit{energy measure} $\wt{\mu}_{\< f\>}$ as the unique Radon measure on $\mathcal X$ satisfying 
\begin{equation}\label{e:3.1}
\int_\mathcal X g(z)\wt{\mu}_{\< f\>}(dz)=\wt{\mcE}(f,fg)-\frac12\wt{\mcE}(f^2,g)\quad  \hbox{ for all }g\in C_c(\mathcal{X})\cap \wt{\mcF}.
\end{equation}
For general $f\in\wt{\mcF}$, we define $\wt{\mu}_{\< f\>}=\lim\limits_{n\to\infty}\wt{\mu}_{\< f_n\>}$, where $f_n=(f\wedge n)\vee (-n)$; this limit is known to exist (see \cite{CF, FOT}). Note that our convention differs from that in \cite{FOT} by a constant factor of $\frac{1}{2}$,
so that $\wt{\mu}_{\< f\>}(\mathcal X)=\wt{\mcE}(f)$ for $f\in \wt{\mcF}$.\medskip 

 In the rest of this paper, we write $\mu_{\<\cdot\>}$ for the energy measure associated with 
 the strongly local,  symmetric regular Dirichlet form $(\mcE,\mcF)$ on $ L^2(F; \mu)$,
   and write $\mu_{B,\<\cdot\>}$ for the energy measure associated with $(\mcE_B,\mcF)$.

\begin{lemma}\label{lemma31}
Let $n\geq 0$, $w\in W_n$ and $f\in\mcF$.
\begin{enumerate}
	\item[(a).] $\mu_{\<f\>}(\partial_oF^w)=0$.  
	\item[(b).] $\mu_{\<f\>}(F^w)=\mcE^w(f|_{F^w})$. 
\end{enumerate}
\end{lemma}
\begin{proof}
(a) is an immediate consequence of \cite[Theorem 2.9]{Murugan} by noticing that $F^w\setminus \partial_oF^w$ is a uniform domain by \cite[Proposition 2.4]{CQ2}. \medskip

(b). Let $\mu^w_{\<\cdot\>}$ denote the energy measure associated with $(\mcE^w,\mcF^w)$. For each $f\in\mcF$ with compact support in  $F^w$, it holds by the local symmetry of $(\mcE,\mcF)$ that 
\begin{align*} 
\int_{F^w} g(z)\mu_{\< f\>}(dz)
&=\mcE(f,fg)-\frac12\mcE(f^2,g)\\
&=\sum_{v\in W_n}\big(\mcE^v(f,fg)-\frac12\mcE^v(f^2,g)\big)\\
&=\mcE^w(f,fg)-\frac12\mcE^w(f^2,g)=\int_{F^w} g(z)\mu^w_{\<f|_{F^w}\>}(dz)
\end{align*}
for   every $g\in\mcF$ supported on $F^w$. 
In particular, this implies 
\begin{equation}\label{e:3.1*} 
\mu_{\<f\>}(F^w\setminus\partial_o F^w)=\mu^w_{\<f|_{F^w}\>}(F^w\setminus\partial_oF^w).
\end{equation}
Summing over $w\in W_n$ and using (a) together with the local symmetry of $(\mcE,\mcF)$, we have
\begin{equation}\label{e:3.2} 
 \mcE(f)=\sum_{w\in W_n}\mu_{\<f\>}(F^w)
=\sum_{w\in W_n}\mu^w_{\<f|_{F^w}\>}(F^w\setminus\partial_oF^w)
\leq \sum_{w\in W_n}\mcE^w(f|_{F^w})=\mcE(f).
 \end{equation}
Combining \eqref{e:3.1*} and \eqref{e:3.2} yields (b). 
\end{proof}

\begin{proposition}\label{prop32}
Let $n\geq 0$.
\begin{enumerate}
	\item[(a)] For $w\in W_n$, it holds that 
	$\mcF=\{f\circ\Psi_w:f\in\mcF^w\}$.
	
	\item[(b)] For  $w,v\in W_n$,
	\[
	\mcE^w(f\circ\Psi^{-1}_w)=\mcE^v(f\circ\Psi^{-1}_v)\quad\hbox{ for }f\in\mcF.
	\]
	
	\item[(c)] By (b), we can define $\mcE^{(n)}$ by 
	\[
	\mcE^{(n)}(f,g):=\mcE^{w}(f\circ\Psi_w^{-1},g\circ\Psi_w^{-1})\quad\hbox{ for }f,g\in\mcF,
	\] 
	where we arbitrarily fix a $w\in W_n$. In particular,  $(\mcE^{(n)},\mcF)$ is a Dirichlet form satisfying ${\bf HK}(d_w)$, local symmetry and reflection symmetry. 
\end{enumerate}
\end{proposition}

\smallskip

\begin{remark}\label{remark33} \rm 
Note that for the standard Dirichlet form $(\mcE_B,\mcF)$, the self-similar property  \cite[Theorem 3.2]{BBKT} implies that 
\[
\mcE_B^{(n)}(f)=3^{n(d_w-d_f)}\mcE_B(f)\quad\hbox{ for }f\in\mcF.
\]
 \end{remark}

\begin{proof}
(a) By the self-similar property of $(\mcE_B,\mcF)$, we see $g|_{F^w}\circ\Psi_w\in\mcF$ for any $g\in\mcF$. In other words, by the definition of $\mcF^w$, 
\[
\mcF\supset \{f\circ\Psi_w:f\in\mcF^w\}. 
\]
On the other hand, by \cite[Proposition 5.1]{Hino}, for any $g\in\mcF$, it holds that $g\circ\Psi^{-1}_w\circ\varphi_w\in\mcF$, which implies $g\circ\Psi^{-1}_w\in\mcF^w$. Therefore,
\[
\mcF\subset \{f\circ\Psi_w:f\in\mcF^w\}. 
\]

(b)  In fact, the functions $f\circ\Psi_w^{-1}\circ\varphi_w$ and $f\circ\Psi_v^{-1}\circ\varphi_v$ either coincide or differ by a reflection symmetry, specifically, by $\mcG_x$, $\mcG_y$ or their composition  $\mcG_x\circ\mcG_y$. Thus, (b) follows from the reflection symmetry of $(\mcE,\mcF)$ and the definition of $\mcE^w$ and $\mcE^v$. \medskip

(c)  First, by Remark \ref{remark33} and the fact that $C_1\mcE_B \leq \mcE\leq C_2\mcE_B$ for some positive constants $C_1,C_2$ (Remark \ref{remark24}), we have
\[
C_13^{n(d_w-d_f)}\mcE_B\leq \mcE^{(n)}\leq C_23^{n(d_w-d_f)}\mcE_B.
\]
Therefore, by the stability theorem of sub-Gaussian heat kernel estimates (see \cite{AB,GHL2}), the form $(\mcE^{(n)},\mcF)$ satisfies the condition ${\bf HK}(d_w)$.

Next, we verify the local symmetry of $(\mcE^{(n)},\mcF)$. For any $f\in\mcF$   and $v\in W_m$ with $m\geq 0$, 
\begin{align*} 
\mcE^{(n),v}(f|_{F^v})&=8^{-m}\mcE^{(n)}(f|_{F^v}\circ\varphi_v)\\
&=8^{-m}\mcE^w(f|_{F^v}\circ\varphi_v\circ\Psi_w^{-1})\\
&=8^{-m-n}\mcE(f|_{F^v}\circ\varphi_v\circ\Psi_w^{-1}\circ\varphi_{w})\\
&=8^{-m-n}\mcE(f|_{F^v}\circ\Psi_w^{-1}\circ\varphi_{w\cdot v})\\
&=\mcE^{w\cdot v}(f|_{F^v}\circ\Psi_w^{-1}).
\end{align*}
On the other hand, note that 
\begin{align*} 
\mcE^{(n)}(f)=\mcE^{w}(f\circ\Psi_w^{-1}),
\end{align*}
and by Lemma \ref{lemma31},
\[
\mcE^w(f\circ\Psi_w^{-1})=\sum_{v\in W_m}\mcE^{w\cdot v}(f|_{F^v}\circ\Psi_w^{-1}).
\]
It follows that $\mcE^{(n)}(f)=\sum_{v\in W_m}\mcE^{(n),v}(f|_{
F^v
})$, which establishes the local symmetry  of $\mcE^{(n)}$.

Finally, we check the reflection symmetry. For any $f\in\mcF$,
\begin{align*} 
\mcE^{(n)}(f\circ\mcG_x)&=\mcE^w(f\circ\mcG_x\circ\Psi_w^{-1})\\
&=8^{-n}\mcE(f\circ\mcG_x\circ\Psi_w^{-1}\circ\varphi_w)\\
&=8^{-n}\mcE(f\circ\Psi_w^{-1}\circ\varphi_w\circ\mcG_x)\\
&=8^{-n}\mcE(f\circ\Psi_w^{-1}\circ\varphi_w)\\
&=\mcE^w(f\circ\Psi_w^{-1})=\mcE^{(n)}(f),
\end{align*}

and similarly $\mcE^{(n)}(f\circ\mcG_y)=\mcE^{(n)}(f)$. Therefore, the reflection symmetry of $\mcE^{(n)}$ holds. 
\end{proof}

To conclude this section, we present some basic facts that will be  frequently used later. 

\begin{lemma}\label{lemma35}
Let $n,k\in \mathbb{N}$.
\begin{enumerate}
	\item[(a)] $\mcE^{(n),w}(f)=\mcE^{(n+k)}(f\circ\Psi_w)$ for  $w\in W_k$ and $f\in\mcF^w$.
	\item[(b)] $\mcE^{w}(f\circ\varphi_v|_{F^w})=\mcE^{v}(f|_{F^v})$ for $w,v\in W_n$ and $f\in\mcF$. 
	\item[(c)] $\sum_{v\in A}\mcE^{w}(f\circ\varphi_v|_{F^w},g)=\mcE^A(f|_{F^A},g\circ\varphi_w|_{F^A})$ for  $w\in W_n$, $A\subset W_n$, $f\in \mcF^A$ and $g\in\mcF^w$.  
\end{enumerate}
\end{lemma}
\begin{proof}
(a) We fix an arbitrary $v\in W_n$. Then 
\begin{align*} 
\mcE^{(n),w}(f)
&=8^{-k}\mcE^{(n)}(f\circ\varphi_w)\\
&=8^{-k}\mcE^{v}(f\circ\varphi_w\circ\Psi_v^{-1})\\
&=8^{-(n+k)}\mcE(f\circ\varphi_w\circ\Psi_v^{-1}\circ\varphi_v)=8^{-(n+k)}\mcE(f\circ\Psi_w\circ\Psi_{vw}^{-1}\circ\varphi_{vw})\\
&=\mcE^{vw}(f\circ\Psi_w\circ\Psi_{vw}^{-1})\\
&=\mcE^{(n+k)}(f\circ\Psi_w),
\end{align*} 
where in the fourth line, the two functions are the same as they  both have local symmetry and coincide on the cell $F^{vw}$. \medskip

(b) $\mcE^w(f\circ\varphi_v|_{F^w})=8^{-n}\mcE(f\circ\varphi_v\circ\varphi_w)=8^{-n}\mcE(f\circ\varphi_v)=\mcE^v(f|_{F^v})$. \medskip

(c) $\displaystyle\sum_{v\in A}\mcE^w(f\circ\varphi_v|_{F^w},g)=8^{-n}\sum_{v\in A}\mcE(f\circ\varphi_v,g\circ\varphi_w)=\sum_{v\in A}\mcE^v(f|_{F^v},g\circ\varphi_w|_{F^v})$, where the first equality uses the relation $f\circ\varphi_v\circ\varphi_w=f\circ\varphi_v$, and the second follows from $g\circ\varphi_w\circ\varphi_v=g\circ\varphi_w$.  
\end{proof}

\section{Resistances and extreme ratios}\label{sec4}

Recall that  $(\mcE,\mcF)$   is a symmetric strongly local regular Dirichlet form on $L^2(F; \mu)$ 
   satisfying ${\bf HK}(d_w)$, local symmetry and reflection symmetry. The form $(\mcE^{(n)},\mcF)$ with $n\geq 0$ is defined as in Proposition \ref{prop32}-(c). For any disjoint closed sets $A,B\subset F$, the effective resistance  with respect to $(\mcE^{(n)},\mcF)$ is defined as  
\[
R^{(n)}(A,B):=\left(\inf\Big\{\mcE^{(n)}(f):\,f\in\mcF,\,f|_A=0\hbox{ and }f|_B=1\Big\}\right)^{-1}. 
\]
In this and the next section, we consider the following specific effective resistances: 
\begin{align*}
R_x^{(n)}:=R^{(n)}(L_2,L_4),\ R_y^{(n)}:=R^{(n)}(L_1,L_3)\ \hbox{ and }\ R^{(n)}(S_1,S_2),
\end{align*}
where, for a fixed large integer $k\geq2$,
\[
S_1:=\big([0,\frac13]\times\{\frac13+\frac{3^{-k}}2\}\big)\cap F
\ \hbox{ and }\ S_2:=\big(\{\frac13+\frac{3^{-k}}2\}\times [0,\frac13]\big)\cap F. \medskip 
\]

Our main goal in this section is to understand the following two extreme ratios: 
\begin{align*} 
\sup(\mcE|\mcE_B)&:=\sup\left\{\frac{\mathcal{E}(f)}{\mathcal{E}_B(f)}:f\in \mcF,\, f\hbox{ is not a constant}\right\},\\ \inf(\mcE|\mcE_B)&:=\inf\left\{\frac{\mathcal{E}(f)}{\mathcal{E}_B(f)}:f\in \mcF,\, f\hbox{ is not a constant}\right\}.
\end{align*}
The following proposition establishes a key relationship between these extreme ratios and the effective resistances. 

\begin{proposition}\label{prop41}
\begin{enumerate}
	\item[(a)] There exist constants $C_1,C_2>0$, independent of $\mcE$, such that 
	\[
	C_1\cdot\limsup_{n\to\infty}\frac{3^{n(d_f-d_w)}}{R_x^{(n)}\wedge R_y^{(n)}}\leq\sup(\mcE|\mcE_B)\leq C_2\cdot\liminf_{n\to\infty}\frac{3^{n(d_f-d_w)}}{R_x^{(n)}\wedge R_y^{(n)}}.
	\]
	
	\item[(b)] There exist constants $C_3,C_4>0$, depending on $k$ but independent of $\mcE$, such that 
	\[
	C_3\cdot\limsup_{n\to\infty}\frac{3^{n(d_f-d_w)}}{R^{(n)}(S_1,S_2)}
	\leq\inf(\mcE|\mcE_B)\leq C_4\cdot\liminf_{n\to\infty}\frac{3^{n(d_f-d_w)}}{R^{(n)}(S_1,S_2)}. 
	\]
\end{enumerate}
\end{proposition} 

\smallskip

\begin{remark}\label{remark42} \rm
Let $R^{(n)}_B$ denote the effective resistance with respect to $(\mcE^{(n)}_B,\mcF)$. Since $\mcE^{(n)}_B=3^{n(d_w-d_f)}\mcE_B$ for all $n\geq 0$, $R^{(n)}_B(L_2,L_4)=R^{(n)}_B(L_1,L_3)=c_1\cdot3^{n(d_f-d_w)}$ and $R^{(n)}_B(S_1,S_2)=c_2\cdot3^{n(d_f-d_w)}$ for some positive constants $c_1,c_2$.  It follows immediately that for $n\geq 0$,
\begin{align*}
\sup(\mcE|\mcE_B)&\geq\max \bigg\{c_1\frac{3^{n(d_f-d_w)}}{R^{(n)}_x}, \,  c_1\frac{3^{n(d_f-d_w)}}{R^{(n)}_y},  \, 
c_2\frac{3^{n(d_f-d_w)}}{R^{(n)}(S_1,S_2)}\bigg\},\\
\inf(\mcE|\mcE_B)&\leq \min\bigg\{ c_1\frac{3^{n(d_f-d_w)}}{R^{(n)}_x}, \, c_1\frac{3^{n(d_f-d_w)}}{R^{(n)}_y}, \,  c_2\frac{3^{n(d_f-d_w)}}{R^{(n)}(S_1,S_2)}
\bigg\}.
\end{align*} 
So, to prove Proposition \ref{prop41}, it suffices to show 
\[
\sup(\mcE|\mcE_B)\leq C\cdot\liminf_{n\to\infty}\frac{3^{n(d_f-d_w)}}{R_x^{(n)}\wedge R_y^{(n)}}\ \hbox{ and }\ \inf(\mcE|\mcE_B)\geq C'\cdot\limsup_{n\to\infty}\frac{3^{n(d_f-d_w)}}{R^{(n)}(S_1,S_2)}
\]
for some $C,C'>0$ independent of $\mcE$.
\end{remark}

\subsection{Proof of Proposition \ref{prop41}-(a)}
The proof of Proposition \ref{prop41}-(a) follows a standard extension argument. we  introduce the following notations. 

\begin{enumerate}
\item Let 
\[
V_0=\{q_1,q_2,q_3,q_4\}
\] 
be the four corner vertices of the unit cube $F_0$,
 and for $n\geq 1$, let 
\[
V_n=\bigcup_{w\in W_n}\Psi_w(V_0). 
\] 
be the set of \textit{level-$n$ vertices}. 

\smallskip

\item Let $h^{(n)}_x$ and $h^{(n)}_y$ be functions in $\mcF$ such that 
\begin{align*}
&h^{(n)}_x|_{L_4}=1,\quad h^{(n)}_x|_{L_2}=0 \quad \hbox{and} \quad \mcE^{(n)}(h^{(n)}_x)=1/R^{(n)}_x,\\
&h^{(n)}_y|_{L_1}=1,\quad h^{(n)}_y|_{L_3}=0 \quad  \hbox{and} \quad \mcE^{(n)}(h^{(n)}_y)=1/R^{(n)}_y.
\end{align*}
Define
\[
h^{(n)}_1=h_x^{(n)}\wedge h_y^{(n)}.
\]
Note that 
\begin{align*} &h_1^{(n)}(q_1)=1,\quad  h_1^{(n)}|_{L_1}=h_x^{(n)}|_{L_1}, \quad  h_1^{(n)}|_{L_4}=h_y^{(n)}|_{L_4},\\& h_1^{(n)}|_{L_2\cup L_3}=0
\quad \hbox{ and } \quad h_1^{(n)}|_{F\cap[0,\frac12]\times [0,\frac12]}\geq 1/2.
\end{align*}
Define 
\[
h^{(n)}_2=h^{(n)}_1\circ\mcG_x,\quad h^{(n)}_3=h^{(n)}_1\circ\mcG_x\circ\mcG_y,
\quad \hbox{and}\quad h^{(n)}_4=h^{(n)}_1\circ\mcG_y.
\]
Set 
\[
u^{(n)}_i=\frac{h^{(n)}_i}{\sum_{j=1}^4 h_j^{(n)}} \quad \hbox{ for } i=1,2,3,4,
\]
Note that $\sum_{j=1}^4 h_j^{(n)}|_{\partial_o F}=1$ and $\sum_{j=1}^4 h_j^{(n)}\geq \frac 12$ on $F$.
\end{enumerate}

\begin{lemma}\label{lemma43}
$\mcE^{(n)}\big(u_i^{(n)}\big)\leq \frac{648}{R^{(n)}_x\wedge R^{(n)}_y}$.
\end{lemma}

\begin{proof}
By the reflection symmetry of $\mcE^{(n)}$ from Proposition \ref{prop32}-(c),  
\[
\mcE^{(n)}(h_i^{(n)})=\mcE^{(n)}(h_1^{(n)})\leq \mcE^{(n)}(h^{(n)}_x)+\mcE^{(n)}(h^{(n)}_y)=1/R^{(n)}_x+1/R^{(n)}_y.
\]
Define  $\varphi (t):= 1/t$ for $t\geq 1$ and $\varphi (t) :=t$ for $t< 1$, which is a normal contraction on $\R$.
As $2\sum_{j=1}^4h^{(n)}_j\geq 1$, we have by the Markovian property of $\mcE^{(n)}$, 
\begin{align*}
\mcE^{(n)}\big(1/\sum_{j=1}^4h^{(n)}_j\big)&=4\mcE^{(n)}\big(\varphi(2\sum_{j=1}^4h^{(n)}_j)\big)\leq 4\mcE^{(n)}(2\sum_{j=1}^4h_j^{(n)})\leq 16\mcE^{(n)}\big(\sum_{j=1}^4h_j^{(n)}\big)\\& \leq16\cdot 16\mcE^{(n)}(h_1^{(n)})\leq 256/R^{(n)}_x+256/R^{(n)}_y,
\end{align*}
  It follows that
\begin{align*} 
\mcE^{(n)}(u^{(n)}_i)&=\mcE^{(n)}\big(h^{(n)}_i/\sum_{j=1}^4h^{(n)}_j\big)\\&\leq \Big(\|h^{(n)}_i\|_\infty\sqrt{\mcE^{(n)}\big(1/\sum_{j=1}^4h^{(n)}_j\big)}+\big\|1/\sum_{j=1}^4h_j^{(n)}\big\|_\infty\sqrt{\mcE^{(n)}(h^{(n)}_i)}\Big)^2\\&\leq 324/R^{(n)}_x+324/R^{(n)}_y,
\end{align*} 
where the first inequality holds by \cite[Theorem 1.4.2 (ii)]{FOT}, and the second follows from  $\|h^{(n)}_i\|_\infty= 1$ and $\|1/\sum_{j=1}^4h_j^{(n)}\|_\infty\leq 2$.
\end{proof}

We now proceed to prove Proposition \ref{prop41}-(a). Fix $f\in \mcF$, and define $f_n\in l(V_n)$ by 
\[
f_n(q):=\frac1{\mu\big(B(q,3^{-n})\big)}\int_{B(q,3^{-n})}fd\mu\quad\hbox{ for each }q\in V_n,
\]
where $B(q,3^{-n})$ denotes the ball of radius $3^{-n}$ centered at $q$.
\begin{lemma}\label{lemma44}
There exists a constant $C>0$ such that 
\[
\sum_{w\in W_n}\sum_{i=1}^4\Big(f_n(\Psi_wq_i)-\frac14\sum_{j=1}^4f_n(\Psi_wq_j)\Big)^2\leq C\cdot 3^{n(d_f-d_w)}\mcE_B(f)\quad\hbox{ for any }f\in\mcF.
\]
\end{lemma}

\begin{proof}
As a consequence of ${\bf HK}(d_w)$, it is well-known that the Poincar\'e inequality for $(\mcE_B,\mcF)$ (see, for example, \cite[Theorem 7.3]{BB3}) holds: 
\[
\frac{1}{\mu\big(B(z,c_1r)\big)}\int_{B(z,c_1r)\times B(z,c_1r)}\big(f(z_1)-f(z_2)\big)^2\mu(dz_1)\mu(dz_2)\leq C_2r^{d_w}\mu_{B,\<f\>}\big(B(z,r)\big),
\] for any $z\in F$ and $r>0$,
where $\mu_{B,\<f\>}$ denotes the energy measure of $f$ associated with $(\mcE_B,\mcF)$, and $c_1\in(0,1),C_2\in(0,\infty)$ are constants independent of $f$. For $n\geq 0$ and $w\in W_n$, we apply the inequality to the ball $B(\Psi_wq_1, C_33^{-n})$ with $C_3=(1+\sqrt{2})c_1^{-1}$ to obtain for $i\neq j\in\{1,2,3,4\}$ that 
\begin{align*} 
&\quad\ \Big(f_n(\Psi_wq_i)-f_n(\Psi_wq_j)\Big)^2\\
&\leq \frac1{\mu\big(B(\Psi_wq_i,3^{-n})\big)\mu\big(B(\Psi_wq_j,3^{-n})\big)}\int_{B(\Psi_wq_i,3^{-n})\times B(\Psi_wq_j,3^{-n})}\big(f(z_1)-f(z_2)\big)^2\mu(dz_1)\mu(dz_2)\\
&\leq C_43^{2nd_f}\int_{B\big(\Psi_wq_1,(1+\sqrt{2})3^{-n}\big)\times B\big(\Psi_wq_1,(1+\sqrt{2})3^{-n}\big)}\big(f(z_1)-f(z_2)\big)^2\mu(dz_1)\mu(dz_2)\\
&\leq C_53^{n(d_f-d_w)}\mu_{B,\<f\>}\big(B({\Psi_wq_1},C_33^{-n})\big),
\end{align*}
where $C_4,C_5$ are some positive constants that do not depend on $f$ and $n$. It follows
\begin{align*} 
\sum_{i=1}^4\Big(f_n(\Psi_wq_i)-\frac14\sum_{j=1}^4f_n(\Psi_wq_j)\Big)^2&\leq \sum_{i=1}^4\max_{j\neq i}\big(f_n(\Psi_wq_i)-f_n(\Psi_wq_j)\big)^2\\
&\leq 4C_53^{n(d_f-d_w)}\mu_{B,\<f\>}\big(B({\Psi_wq_1},C_33^{-n})\big) .
\end{align*}
The desired inequality follows by summing the above estimate over all $w\in W_n$. 
\end{proof}

\begin{proof}[Proof of Proposition \ref{prop41}-(a)]
Let $f\in\mcF$ be fixed, and let $f_n$ be defined as in Lemma \ref{lemma44}. We define a sequence of approximating functions $\{g_n; n\geq 0\}$ by 
\[
g_n=\sum_{w\in W_n}\sum_{i=1}^4 f_n(\Psi_wq_i)\cdot \big(u_i^{(n)}\circ\Psi_w^{-1}\big).
\]
Note that for each $w$ and $i$, $g_n|_{\Psi_w L_i}$ depends only on the values of $f_n$ at the two endpoints of $\Psi_wL_i$, so $g_n$ is continuous. By \cite[Proposition 5.1]{Hino}, we have $g_n\in\mcF$. Moreover, for each $w\in W_n$,
\begin{align*}
\mcE^w(g_n|_{F^w})&=\mcE^{(n)}(g_n\circ \Psi_w)\\
&=\mcE^{(n)}\big(g_n\circ \Psi_w-\frac14\sum_{j=1}^4f_n(\Psi_wq_j)\big)\\
&=\mcE^{(n)}\Big(\sum_{i=1}^4\big(f_n(\Psi_wq_i)-\frac14\sum_{j=1}^4f_n(\Psi_wq_j)\big)u_i^{(n)}\Big)\\
&\leq 4\sum_{i=1}^4\big(f_n(\Psi_wq_i)-\frac14\sum_{j=1}^4f_n(\Psi_wq_j)\big)^2\mcE^{(n)}(u_i^{(n)})\\
&\leq \frac{2592}{R^{(n)}_x\wedge R^{(n)}_y}\sum_{i=1}^4\big(f_n(\Psi_wq_i)-\frac14\sum_{j=1}^4f_n(\Psi_wq_j)\big)^2,
\end{align*}
where the first equality follows from Proposition \ref{prop32}-(c), the third equality from the identity $\sum_{i=1}^4u_i^{(n)}=1$, and the last inequality uses Lemma \ref{lemma43}. Now, by Lemmas \ref{lemma31} and \ref{lemma44}, 
\begin{align}
\mcE(g_n)&=\sum_{w\in W_n}\mcE^w(g_n|_{F^w})\nonumber\\
&\leq \frac{2592}{R^{(n)}_x\wedge R^{(n)}_y}\sum_{w\in W_n}\sum_{i=1}^4\Big(f_n(\Psi_wq_i)-\frac14\sum_{j=1}^4f_n(\Psi_wq_j)\Big)^2\label{e:prop4.1.1}\\
&\leq C\frac{3^{n(d_f-d_w)}}{R_x^{(n)}\wedge R_y^{(n)}}\mcE_B(f)\nonumber
\end{align} 
for some constant $C>0$.
By construction,  $g_n\to f$ uniformly, and thus in $L^2(F;\mu)$, as $n\to \infty$. Hence, by the lower-semicontinuity of $\mcE$,  
\[
\mcE(f)\leq \liminf_{n\to\infty}\mcE(g_n)\leq C\liminf_{n\to\infty}\frac{3^{n(d_f-d_w)}}{R_x^{(n)}\wedge R_y^{(n)}}\mcE_B(f). 
\]
Since this inequality holds for every  $f\in \mcF$, Proposition \ref{prop41}-(a) follows. 
\end{proof}

\subsection{Building brick functions}\label{sec4.3}
Before proving Proposition \ref{prop41}-(b), we introduce some auxiliary functions which behave like flows between faces. We always fix a positive integer $k\geq 2$ for the definition of $S_1$ and $S_2$. \medskip 

First, for our construction, we consider a function $f^{(n)}\in\mcF$ with $n\geq 0$ such that 
\begin{align*} 
&f^{(n)}=0\quad \hbox{ on }S_1\cup \mcG_y(S_1)\cup \mcG_x(S_1)\cup \mcG_x\circ\mcG_y(S_1),\\ 
&f^{(n)}=1\quad\hbox{ on }S_2\cup \mcG_y(S_2)\cup \mcG_x(S_2)\cup \mcG_x\circ\mcG_y(S_2),\\
f^{(n)}\hbox{ is }\mcE^{(n)}&\hbox{-harmonic in }F\setminus \big(S_1\cup S_2\cup \mcG_x(S_1\cup S_2)\cup \mcG_y(S_1\cup S_2)\cup\mcG_x\circ\mcG_y(S_1\cup S_2)\big). 
\end{align*}
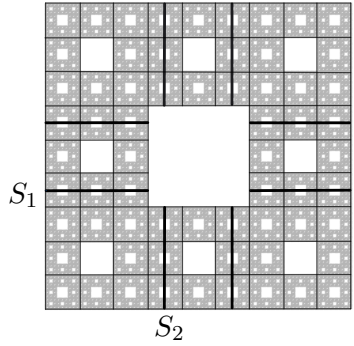
\begin{figure}
    \centering
    \includegraphics[width=0.28\linewidth]{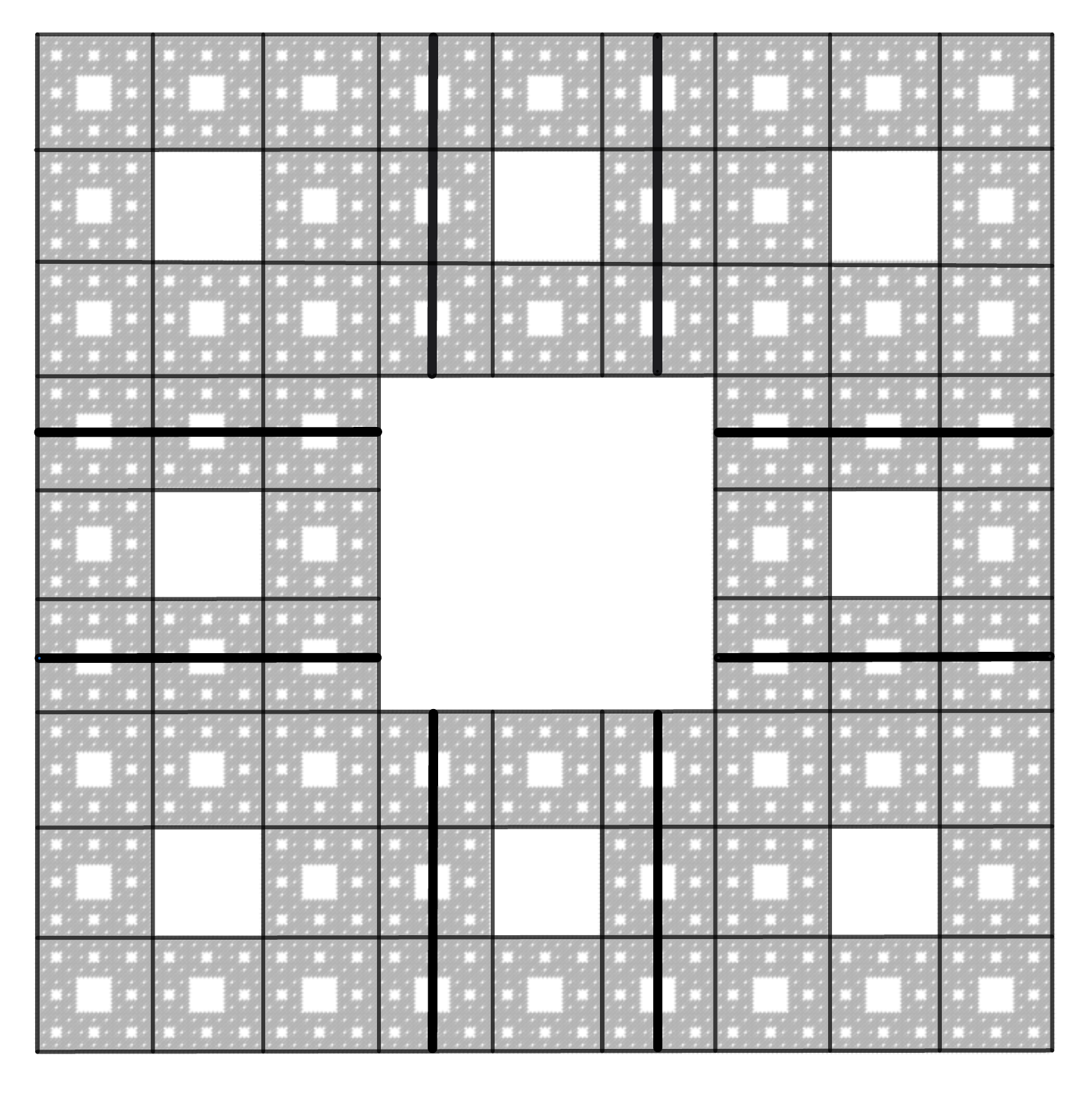}
    \begin{picture}(0,0)
\put(-137,45){$S_1$}\put(-82,-5){$S_2$}
\end{picture}
    \caption{$S_1$, $S_2$ and their reflections with respect to $\mcG_x$, $\mcG_y$.}
    \label{Fig3}
\end{figure}
See Figure \ref{Fig3} for an illustration. Note that by symmetry,
\[
1/R^{(n)}(S_1, S_2)\leq\mcE^{(n)}(f^\bn)\leq 4/R^{(n)}(S_1, S_2).
\]
We renormalize $f^\bn$ by taking
\[
g^\bn=4f^\bn/\mcE^{(n)}(f^{(n)})
\]
so that
\[
\mcE^{(n)}(g^\bn,f^\bn)=4.
\]
Heuristically, $g^\bn$ represents the unit flow from $S_2$ to $S_1$ within a corner of $F$, along with three other reflected copies of it. \medskip 

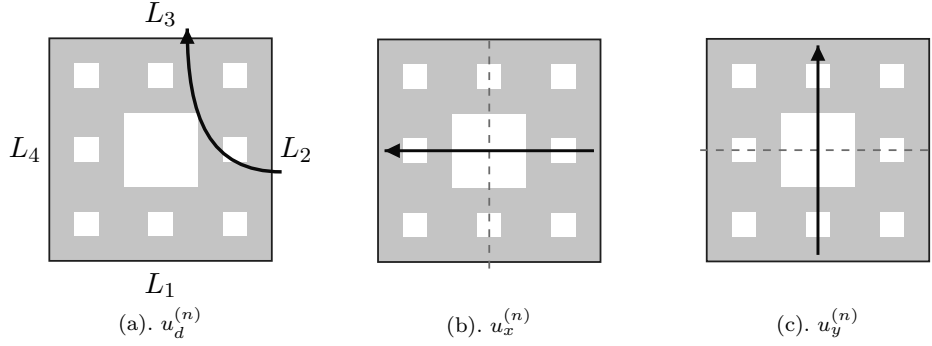
\begin{figure}[H]
	\centering
	\begingroup
	\newcommand{\flowpanel}[3]{%
		\begin{minipage}[c]{112pt}
			\centering
			{\setlength{\unitlength}{1pt}%
				\begin{picture}(112,112)
					\put(8,8){\includegraphics[width=96pt]{#1.pdf}}
					#3
			\end{picture}}\\[-1pt]
			{\scriptsize #2}
	\end{minipage}}
	\flowpanel{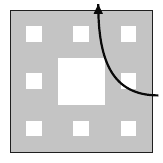}{(a).~$u_d^{(n)}$}{%
		\put(56,5){\makebox(0,0){$L_1$}}%
		\put(107,56){\makebox(0,0){$L_2$}}%
		\put(56,107){\makebox(0,0){$L_3$}}%
		\put(5,56){\makebox(0,0){$L_4$}}}%
	\hspace{8pt}
	\flowpanel{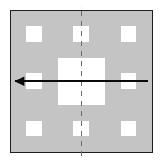}{(b).~$u_x^{(n)}$}{}\hspace{8pt}
	\flowpanel{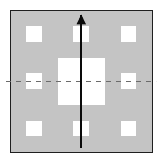}{(c).~$u_y^{(n)}$}{}
	\endgroup
	\caption{The building brick functions $u_d^{(n)}$, $u_x^{(n)}$, and $u_y^{(n)}$.} 
	\label{Fig5}
\end{figure}

Next, we use folding maps to construct three basic \textit{building brick functions} $u^{(n)}_d,u^{(n)}_x,u^{(n)}_y$. Heuristically, with respect to $\mcE^{(n)}$, $u^{(n)}_d$ represents a unit flow from $L_2$ to $L_3$; $u^{(n)}_x$ represents a unit flow from $L_2$ to $L_4$; and $u^{(n)}_y$ represents a unit flow from $L_1$ to $L_3$. See Figure \ref{Fig5} for an illustration.

\begin{enumerate}
\item We define 
\[
u^\bn_d=\sum_{v\in W_{k-1}}g^\bn\circ\Psi_1\circ\varphi_v\circ\Psi_{\dot{3}^{k-1}}.  
\]
In other words, we first use $\Psi_1$ to zoom in $g^{(n)}|_{\Psi_1(F)}$, then unfold $g^{(n)}\circ\Psi_1$ for each $v$ and sum over them, and finally  zoom in on the corner cell $F^{\dot{3}^{k-1}}$ to obtain $u^{(n)}_d$.

\item We define $u^\bn_x$ in two steps. First,  define 
\[
\
\wt{u}^\bn_x=\sum_{v\in \{1,4,8\}^{k-1}}g^\bn\circ\Psi_5\circ\varphi_v\circ\Psi_{\dot{4}^{k-1}}.  
\]
Then,  define 
\[
u^\bn_x=\wt{u}^\bn_x-\wt{u}^\bn_x\circ\mcG_x.
\]

\item We define $u^\bn_y$ in two steps. First,  define 
\[
\
\wt{u}^\bn_y=\sum_{v\in \{1,5,2\}^{k-1}}g^\bn\circ\Psi_8\circ\varphi_v\circ\Psi_{\dot{2}^{k-1}}.  
\]
Then,  define 
\[
u^\bn_y=\wt{u}^\bn_y-\wt{u}^\bn_y\circ\mcG_y. \medskip
\]
\end{enumerate}

See Figure \ref{Fig4} for an illustration of the constructions.

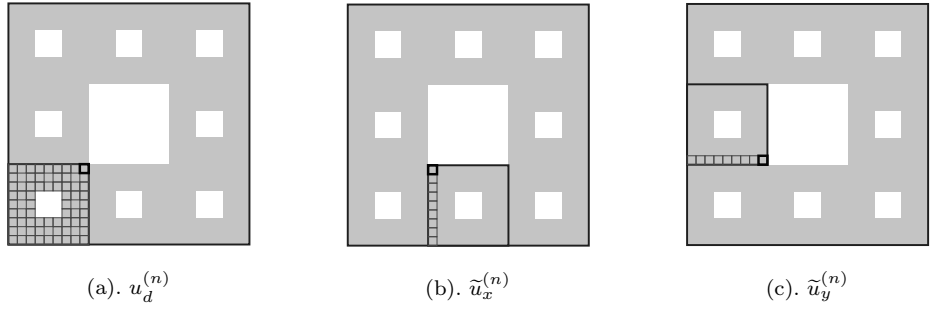
\begin{figure}[H]
	\centering
	\begingroup
	\newcommand{\embeddedpanel}[2]{%
		\begin{minipage}[c]{112pt}
			\centering
			{\setlength{\unitlength}{1pt}%
				\begin{picture}(104,104)
					\put(0,0){\includegraphics[width=104pt]{#1.pdf}}
			\end{picture}}\\[-1pt]
			{\scriptsize #2}
	\end{minipage}}
	\embeddedpanel{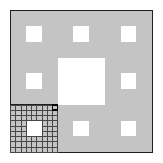}{(a).~$u_d^{(n)}$}\hspace{12pt}
	\embeddedpanel{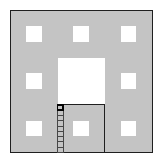}{(b).~$\widetilde u_x^{(n)}$}\hspace{12pt}
	\embeddedpanel{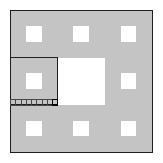}{(c).~$\widetilde u_y^{(n)}$}
	\endgroup
	\caption{Construction of $u_d^{(n)}$, $\widetilde u_x^{(n)}$, and $\widetilde u_y^{(n)}$. We restrict $g^{(n)}$ to the union of small bordered cells, fold the resulting values onto the highlighted cell, and then zoom in to obtain the desired functions.}
	\label{Fig4}
\end{figure}

We first present some basic properties of the functions $u^{(n)}_d$, $u^{(n)}_x$, $\wt{u}^{(n)}_x$, $u^{(n)}_y$, and $\wt{u}^{(n)}_y$ in the following lemmas.

\begin{lemma}\label{lemma45}
The following energy estimates hold:
\begin{align*}
\mcE^{(n+k)}(u^{(n)}_d)\leq 4\cdot 8^{k-1}\,R^{(n)}(S_1,S_2),\\
\mcE^{(n+k)}(\wt{u}^{(n)}_x)\leq 4\cdot 3^{k-1}\,R^{(n)}(S_1,S_2),\\
\mcE^{(n+k)}(\wt{u}^{(n)}_y)\leq 4\cdot 3^{k-1}\,R^{(n)}(S_1,S_2). 
\end{align*}
\end{lemma}
\begin{proof}
First, note that  $\mcE^{(n)}(f^{(n)})\geq 1/R^{(n)}(S_1,S_2)$. It follows that 
\[
\mcE^{(n)}(g^{(n)})=\mcE^{(n)}\big(\frac{4f^{(n)}}{\mcE^{(n)}(f^{(n)})}\big)=\frac{16}{\mcE^{(n)}(f^{(n)})}\leq 16R^{(n)}(S_1,S_2). 
\]
In particular, by Lemma \ref{lemma31}, Proposition \ref{prop32} and using symmetry, we have 
\[
\mcE^{(n),A}(g^{(n)}|_{F^A})\leq 4R^{(n)}(S_1,S_2),
\]
where $A=\big(\{1\}\times W_{k-1})\cup (\{5\}\times\{1,8,4\}^{k-1})\cup (\{8\}\times\{1,5,2\}^{k-1}\big)$. Then, 
\begin{align*}
\mcE^{(n+k)}(u^{(n)}_d)&=\mcE^{(n+k)}\big(\sum_{v\in W_{k-1}}g^\bn\circ\Psi_1\circ\varphi_v\circ\Psi_{\dot{3}^{k-1}}\big)\\
&=\mcE^{(n+1),\dot{3}^{k-1}}\big(\sum_{v\in W_{k-1}}g^{(n)}\circ\Psi_1\circ\varphi_v|_{F^{\dot{3}^{k-1}}}\big)\\
&\leq 8^{k-1}\sum_{v\in W_{k-1}}\mcE^{(n+1),\dot{3}^{k-1}}(g^{(n)}\circ\Psi_1\circ\varphi_v|_{F^{\dot{3}^{k-1}}})\\
&=8^{k-1}\sum_{v\in W_{k-1}}\mcE^{(n+1),v}(g^{(n)}\circ\Psi_1|_{F^v})\\
&=8^{k-1}\sum_{v\in W_{k-1}}\mcE^{(n+k)}(g^{(n)}\circ\Psi_1\circ\Psi_v)\\
&= 8^{k-1}\sum_{v\in W_{k-1}}\mcE^{(n),1v}(g^{(n)}|_{F^{1v}})\leq 8^{k-1}\sum_{w\in A}\mcE^{(n),w}(g^{(n)}|_{F^w})\leq 4\cdot 8^{k-1} R^{(n)}(S_1,S_2),
\end{align*}
 where the third equality uses Lemma \ref{lemma35}-(b), and the second, fourth and fifth equalities use Lemma \ref{lemma35}-(a). The other two estimates follow similarly. 
\end{proof}

\begin{lemma}\label{lemma46}
\begin{enumerate}
	\item[(a)]$u^{(n)}_d$ is $\mcE^{(n+k)}$-harmonic in $F\setminus(L_2\cup L_3)$.  
	\item[(b)]$u^{(n)}_x$ is $\mcE^{(n+k)}$-harmonic in $F\setminus(L_2\cup L_4)$, and $\wt{u}^{(n)}_x$ is constant in $F\cap[1/2,1]\times[0,1]$.
	\item[(c)]$u^{(n)}_y$ is $\mcE^{(n+k)}$-harmonic in $F\setminus (L_1\cup L_3)$, and $\wt{u}^{(n)}_y$ is constant in $F\cap[0,1]\times[1/2,1]$. 
\end{enumerate}
\end{lemma}

\begin{proof}
(a)  For any $f\in C_c\big(F\setminus (L_2\cup L_3)\big)\cap \mcF$, we have 
\begin{equation}\label{e:4.1}
\begin{split} 
\mcE^{(n+k)}(f,u^\bn_d)&=\mcE^{(n+1),\dot{3}^{k-1}}(f\circ\Psi_{\dot{3}^{k-1}}^{-1},u^\bn_d\circ\Psi_{\dot{3}^{k-1}}^{-1})\\
&=\mcE^{(n+1),\dot{3}^{k-1}}\big(f\circ\Psi_{\dot{3}^{k-1}}^{-1},\sum_{v\in W_{k-1}}g^{(n)}\circ\Psi_1\circ\varphi_v|_{F_{\dot{3}^{k-1}}}\big)\\
&=\mcE^{(n+1)}(f\circ\Psi_{\dot{3}^{k-1}}^{-1}\circ\varphi_{\dot{3}^{k-1}},g^{(n)}\circ\Psi_1)\\
&=0,
\end{split} 
\end{equation}
where the first equality follows from Lemma \ref{lemma35}-(a), the second uses the definition of $u^{(n)}_d$, the third is due to Lemma \ref{lemma35}-(c), and the last equality holds since $g^{(n)}\circ\Psi_1$ is $\mcE^{(n+1)}$-harmonic in $F\setminus (L_2\cup L_3)$ and $f\circ\Psi_{\dot{3}^{k-1}}^{-1}\circ\varphi_{\dot{3}^{k-1}}$ vanishes on $L_2\cup L_3$. Therefore, $u^{(n)}_d$ is $\mcE^{(n+k)}$-harmonic in $F\setminus(L_2\cup L_3)$. 
 
\medskip

(b) It is clear from the construction of $f^{(n)}$ that $\wt{u}_x^{(n)}=1$ on $F\cap [1/2,1]\times[0,1]$,  so $\wt{u}_x^{(n)}$ is $\mcE^{(n+k)}$-harmonic in $F\cap (1/2,1)\times[0,1]$.  
By the same argument as in (a), for any $f\in \mcF\cap C_c\big(F\cap (0,1/2)\times [0,1]\big)$, we have 
\[
\mcE^{(n+k)}(f,\wt{u}_x^{(n)})=\mcE^{(n+1),\{1,8,4\}^{k-1}}(f\circ\Psi_{\dot{4}^{k-1}}^{-1}\circ\varphi_{\dot{4}^{k-1}},g^{(n)}\circ\Psi_5)=0.
\]
So $\wt{u}_x^{(n)}$ is also $\mcE^{(n+k)}$-harmonic in $F\cap (0,1/2)\times [0,1]$.  
Combining the two observations, $\wt{u}_x^{(n)}$ is $\mcE^{(n+k)}$-harmonic in $F\cap\big((0,1/2)\cup (1/2,1)\big)\times[0,1]$. 

Next, it follows immediately that $u_x^{(n)}$ is $\mcE^{(n+k)}$-harmonic in $F\cap \big((0,1/2)\cup (1/2,1)\big)\times[0,1]$, so that 
\[
\mcE^{(n+k)}(f-f\circ\mcG_x,u_x^\bn)=0\quad \hbox{ for any }
f\in \mcF\cap C_c\big(F\setminus(L_2\cup L_4)\big). 
\]
Furthermore, since  $u_x^\bn$ is antisymmetric with respect to $\mcG_x$, for any $f\in \mcF\cap C_c\big(F\setminus(L_2\cup L_4)\big)$,
\begin{align*}
    \mcE^{(n+k)}(f+f\circ\mcG_x,u_x^\bn)=\mcE^{(n+k)}(f+f\circ\mcG_x,u_x^\bn\circ\mcG_x)=-\mcE^{(n+k)}(f+f\circ\mcG_x,u_x^\bn), 
\end{align*}
which implies
\[
\mcE^{(n+k)}(f+f\circ\mcG_x,u_x^\bn)=0\quad \hbox{ for any }
f\in \mcF\cap C_c\big(F\setminus(L_2\cup L_4)\big).
\]
Combining the above observations, for any $f\in \mcF\cap C_c\big(F\setminus(L_2\cup L_4)\big)$, we have
\begin{align*}
\mcE^{(n+k)}(f,u_x^\bn)=\frac12\mcE^{(n+k)}(f-f\circ\mcG_x,u_x^\bn)+\frac12\mcE^{(n+k)}(f+f\circ\mcG_x,u_x^\bn)=0.
\end{align*}
Therefore, $u_x^\bn$ is $\mcE^{(n+k)}$-harmonic in $F\setminus(L_2\cup L_4)$.

\smallskip

The proof for (c) is exactly the same as that for (b). 
\end{proof}

Now, we consider the `\textit{normal derivative}' for the building brick functions. For $i=1,2,3,4$, denote by $\check\mcF_i$ the trace of $\mcF$ on $L_i$, i.e. 
\[
\check\mcF_i=\{f|_{L_i}:\,f\in\mcF\}. 
\]
According to \cite[Lemma 5.1]{CF}, $\check\mcF_i$ is a Hilbert space  with the norm  
\[ 
\|u\|_{\check\mcF_i}:=\sqrt{\inf\big\{\mcE_B(f,f):f\in\mcF,f|_{L_i}=u\big\}+\int_{L_i}u^2dz},
\]
where the integral is taken with respect to the one-dimensional Lebesgue measure.

\begin{lemma}\label{lemma47}
There exist bounded linear functionals $\Theta_i:\check{\mcF}_i\to \R$ for $i=1,2,3,4$ such that the following properties hold.
\begin{enumerate}
	\item[(a)] For any $f\in\mcF$,  $\Theta_4(f|_{L_4})=\Theta_2(f\circ\mcG_x|_{L_2})$ and $\Theta_3(f|_{L_3})=\Theta_1(f\circ\mcG_y|_{L_1})$.
	\item[(b)] For any $f\in \mcF$, 
	\begin{align}
		\mcE^{(n+k)}(u^{(n)}_d,f)&=\Theta_2(f|_{L_2})-\Theta_3(f|_{L_3}),\label{lemma4.7.1}\\
		\mcE^{(n+k)}(u^{(n)}_x,f)&=\Theta_2(f|_{L_2})-\Theta_4(f|_{L_4}),\label{lemma4.7.2} \\
		\mcE^{(n+k)}(u^{(n)}_y,f)&=\Theta_1(f|_{L_1})-\Theta_3(f|_{L_3}). \label{lemma4.7.3}
	\end{align}
	\item[(c)] $\Theta_i\1=1$ for $i=1,2,3,4$, where $\1$ denotes the constant function on $ \cup_{j=1}^4 L_j$ with value $1$.  
\end{enumerate}
\end{lemma}

\begin{proof}
We first define the functionals $\Theta_i$ for $i=1,2,3,4$ as follows. 
\begin{itemize}
\item For each $u\in \check\mcF_2$,  choose $f\in \mcF$ such that $f|_{L_2}=u$ and $f|_{L_4}=0$, and  define 
\[
\Theta_2(u)=\mcE^{(n+k)}(u^{(n)}_x,f).
\]
Note that by Lemma \ref{lemma46}-(b), $u^{(n)}_x$ is $\mcE^{(n+k)}$-harmonic in $F\setminus (L_2\cup L_4)$, so the value above is independent of the choice of $f$. 

\item For each $u\in \check\mcF_3$,  choose $f\in \mcF$ such that $f|_{L_3}=u$ and $f|_{L_1}=0$, and  define
\[
\Theta_3(u)=-\mcE^{(n+k)}(u^{(n)}_y,f).
\]
Again,   by Lemma \ref{lemma46}-(c), $u^{(n)}_y$ is $\mcE^{(n+k)}$-harmonic in $F\setminus (L_1\cup L_3)$, so the value above is independent of the choice of $f$.

\item Define $\Theta_4(u)=\Theta_2(u\circ\mcG_x)$ and $\Theta_1(u)=\Theta_3(u\circ\mcG_y)$. 
\end{itemize}

(a) follows immediately from the definition. 

(b). The identities \eqref{lemma4.7.2} and \eqref{lemma4.7.3} are easy consequences of  reflection symmetry. For \eqref{lemma4.7.2},  choose $f_1,f_2\in\mcF$ such that 
\begin{align*}
f_1|_{L_2}=f|_{L_2},\ f_1|_{L_4}=0, \quad\hbox{ and }\quad 
f_2|_{L_2}=0,\ f_2|_{L_4}=f|_{L_4}.
\end{align*}
Then, by the definition of $\Theta_2$, we have
\[
\mcE^{(n+k)}(u^{(n)}_x,f_1)=\Theta_2(f|_{L_2}). 
\] 
By the antisymmetry of $u_x^{(n)}$, the definition of $\Theta_2$ and (a),
\begin{align*} 
\mcE^{(n+k)}(u^{(n)}_x,f_2)&=\mcE^{(n+k)}(u^{(n)}_x\circ\mcG_x,f_2\circ\mcG_x)\\
&=-\mcE^{(n+k)}(u^{(n)}_x,f_2\circ\mcG_x)=-\Theta_2(f_2\circ\mcG_x|_{L_2})=-\Theta_4(f|_{L_4}),
\end{align*} 
Since $u^{(n)}_x$ is $\mcE^{(n+k)}$-harmonic in $F\setminus (L_2\cup L_4)$ and $f-f_1-f_2=0$ on $L_2\cup L_4$, we have 
\[
\mcE^{(n+k)}(u^{(n)}_x,f-f_1-f_2)=0.
\]
Summing these three equalities yields \eqref{lemma4.7.2}. The proof of \eqref{lemma4.7.3} follows by the same argument,
so it is omitted. 

\medskip 

To prove \eqref{lemma4.7.1}, we note that 
\begin{enumerate}
\item[(i)] if $f\in\mcF$ satisfies $f|_{F\cap ([1/2,1]\times [0,1])}=0$, then by \eqref{lemma4.7.2}, Lemma \ref{lemma46}-(b) and the strongly local property
\begin{equation}\label{lemma4.7.4}
\begin{split}
-\Theta_4(f|_{L_4})&=\mcE^{(n+k)}(u^{(n)}_x,f)\\
&=\mcE^{(n+k)}(\wt{u}_x^{(n)},f)-\mcE^{(n+k)}(\wt{u}_x^{(n)}\circ\mcG_x,f)=\mcE^{(n+k)}(\wt{u}_x^{(n)},f) ; 
\end{split}
\end{equation}

\item[(ii)] if $f\in\mcF$ satisfies $f|_{F\cap ([0,1]\times[1/2,1])}=0$, then by \eqref{lemma4.7.3}, Lemma \ref{lemma46}-(c) and the strongly local property
\begin{equation}\label{lemma4.7.5}
\begin{split}
\Theta_1(f|_{L_1})&=\mcE^{(n+k)}(u^{(n)}_y,f)\\
&=\mcE^{(n+k)}(\wt{u}_y^{(n)},f)-\mcE^{(n+k)}(\wt{u}_y^{(n)}\circ\mcG_y,f)=\mcE^{(n+k)}(\wt{u}_y^{(n)},f).
\end{split}
\end{equation}
\end{enumerate}
Now, given $f\in\mcF$, we construct a function $g\in \mcF$ in two steps. 
\begin{enumerate}
\item[Step 1.] Let $g_1=f\circ\Psi_{1\cdot {\dot3}^{k-1}}^{-1}\circ\varphi_{1\cdot {\dot{3}}^{k-1}}$. 

\item[Step 2.] Choose a function $\psi\in\mcF$ such that $\psi=1$ on $\Psi_1(F)$ and $\psi(z)=0$ whenever $d(z,\Psi_1(F))\geq 3^{-k}/2$. For notational simplicity, we require that there are $\psi_x,\psi_y\in\mcF$ such that 
\begin{align*} 
\psi\circ\Psi_w=\psi_x\hbox{ for }w\in\{5\}\times\{1,4,8\}^{k-1},\\
\psi\circ\Psi_w=\psi_y\hbox{ for }w\in\{8\}\times\{1,5,2\}^{k-1}.
\end{align*}
Clearly, such functions $\psi_x$ and $\psi_y$ in $\mcF$ always exist. Define $g=\psi\cdot g_1$. 
\end{enumerate}
Recalling the definition of $g^{(n)}$ from the beginning of this subsection, and noting that $g^{(n)}$ is $\mcE^{(n)}$-harmonic on the support of $g$, we have 
\begin{equation}\label{e:4.5}
\begin{split}
0=&\mcE^{(n)}(g,g^{(n)})\\
=&\sum_{i=1,5,8}\mcE^{(n+1)}(g\circ\Psi_i,g^{(n)}\circ\Psi_i)\\
=&\mcE^{(n+1)}(f\circ\Psi_{\dot3^{k-1}}^{-1}\circ\varphi_{\dot3^{k-1}},g^{(n)}\circ\Psi_1)\\
&+\mcE^{(n+1)}\big((f\circ\mcG_x\circ\Psi_{\dot4^{k-1}}^{-1}\circ\varphi_{\dot4^{k-1}})\cdot (\psi\circ\Psi_5),g^{(n)}\circ\Psi_5\big)
\\&+\mcE^{(n+1)}\big((f\circ\mcG_y\circ\Psi_{\dot2^{k-1}}^{-1}\circ\varphi_{\dot2^{k-1}})\cdot(\psi\circ\Psi_8),g^{(n)}\circ\Psi_8\big)\\
=&\mcE^{(n+k)}(f,u^{(n)}_d)+\mcE^{(n+k)}\big((f\circ\mcG_x)\cdot \psi_x,\wt{u}^{(n)}_x\big)+\mcE^{(n+k)}\big((f\circ\mcG_y)\cdot \psi_y,\wt{u}^{(n)}_y\big)\\
=&\mcE^{(n+k)}(f,u^{(n)}_d)-\Theta_4(f\circ\mcG_x|_{L_4})+\Theta_1(f\circ\mcG_y|_{L_1})\\
=&\mcE^{(n+k)}(f,u^{(n)}_d)-\Theta_2(f|_{L_2})+\Theta_3(f|_{L_3}),
\end{split} 
\end{equation}
where the second equality follows from Lemmas \ref{lemma31} and \ref{lemma35}-(a), the fourth equality follows from the same reasoning as in \eqref{e:4.1}, the fifth equality is due to \eqref{lemma4.7.4} and \eqref{lemma4.7.5}. The desired \eqref{lemma4.7.1} follows immediately. \medskip

(c) We pick three functions $f_d,f_x,f_y\in\mcF$ that satisfy 
\begin{align*} 
&f_x\circ\mcG_x|_{L_2}=f_d|_{L_2}\ \hbox{ and }\  f_x|_{F\cap [1/2,1]\times [0,1]}=1,\\
&f_y\circ\mcG_y|_{L_3}=f_d|_{L_3}\ \hbox{ and }\ 
f_y|_{F\cap [0,1]\times [1/2,1]}=0.
\end{align*}
Then we define a function $f$ on $F^A$ with 
\[
A=\big(\{1\}\times W_{k-1})\cup (\{5\}\times\{1,8,4\}^{k-1})\cup (\{8\}\times\{1,5,2\}^{k-1}\big),
\]
by
\[
f(z)=\begin{cases}
f_d\circ\Psi_{\dot3^{k-1}}^{-1}\circ\varphi_{\dot3^{k-1}}\circ\Psi^{-1}_1(z)\quad&\hbox{ for }z\in \Psi_1(F),\\
f_x\circ\Psi_{\dot4^{k-1}}^{-1}\circ\varphi_{\dot4^{k-1}}\circ\Psi^{-1}_5(z)\quad&\hbox{ for }z\in F^w,w\in \{5\}\times\{1,8,4\}^{k-1},\\
f_y\circ\Psi_{\dot2^{k-1}}^{-1}\circ\varphi_{\dot2^{k-1}}\circ\Psi^{-1}_8(z)\quad&\hbox{ for }z\in F^w,w\in \{8\}\times\{1,5,2\}^{k-1},\\
0&\hbox{ elsewhere}. 
\end{cases}
\]
Recall for a word $w$, $F^w:=\Psi_w(F)$ as defined at the beginning of Section \ref{sec2}.  
Note that  $f$ can be extended to a function in $\mcF$. Recall the function $f^{(n)}$ used in the definition of $g^{(n)}$. Since $g^{(n)}$ is $\mcE^{(n)}$-harmonic in the interior of $\bigcup_{w\in A}F^w$ with respect to $(F,d)$, and since $f,f^{(n)}$ share the same boundary values, by symmetry we see that (by Lemma \ref{lemma31} and Proposition \ref{prop32}) 
\[
\mcE^{(n),A}(f|_{F^A},g^{(n)}|_{F^A})=\frac{1}{4}\mcE^{(n)}(f^{(n)},g^{(n)})=1.
\]
On the other hand, by an argument analogous to  \eqref{e:4.5}, we obtain 
\begin{align*} 
&\quad\ \mcE^{(n),A}(f|_{F^A},g^{(n)}|_{F^A})\\
&=\mcE^{(n+k)}(f_d,u^{(n)}_d)+\mcE^{(n+k)}(f_x,\wt u^{(n)}_x)+\mcE^{(n+k)}(f_y,\wt u^{(n)}_y)\\
&=\mcE^{(n+k)}(f_d,u^{(n)}_d)+\mcE^{(n+k)}(f_x,u^{(n)}_x)+\mcE^{(n+k)}(f_y,u^{(n)}_y)\\
&=\big(\Theta_2(f_d|_{L_2})-\Theta_3(f_d|_{L_3})\big)+\big(\Theta_2(f_x|_{L_2})-\Theta_4(f_x|_{L_4})\big)+\big(\Theta_1(f_y|_{L_1})-\Theta_3(f_y|_{L_3})\big)\\
&=\big(\Theta_2(f_d|_{L_2})-\Theta_3(f_d|_{L_3})\big)+\big(\Theta_2\1-\Theta_2(f_d|_{L_2})\big)+\big(\Theta_3(f_d|_{L_3})-\Theta_3(f_y|_{L_3})\big)\\
&=\Theta_2\1,
\end{align*}
where the second equality follows from the strongly local property of $\mcE^{(n+k)}$ similarly to \eqref{lemma4.7.4} and \eqref{lemma4.7.5}, the third from \eqref{lemma4.7.1}-\eqref{lemma4.7.3}, and the last from the fact $f_y|_{L_3}=0$.
Combining the two expressions yields the desired identity $\Theta_2\1=1$. The identity $\Theta_3\1=1$ follows from the observation that $0=\mcE^{(n+k)}(u^{(n)}_d,\1_F)=\Theta_2\1-\Theta_3\1$ where $\1_F$ is the constant function with value $1$ on $F$. The identity $\Theta_1\1=\Theta_4\1$ follows immediately from the definition.
\end{proof}

\begin{proposition}\label{lemma48}
Let ${\bf c}=(c_1,c_2,c_3,c_4)\in\R^4$ satisfy $c_1+c_2+c_3+c_4=0$. Then there exists  $u_{\bf c}^{(n)}\in \operatorname{span}\{u_d^{(n)},u_x^{(n)},u_y^{(n)}\}$ such that 
\begin{eqnarray}
\label{e:4.9}
&\displaystyle\mcE^{(n+k)}(u^{(n)}_{\bf c},f)=\sum_{i=1}^4c_i\Theta_i(f|_{L_i}) \quad\hbox{ for all }f\in\mcF,\\ 
\label{e:4.10}
&\displaystyle\mcE^{(n+k)}(u^{(n)}_{\bf c})\leq C\cdot R^{(n)}(S_1,S_2)\sum_{i=1}^4c_i^2,
\end{eqnarray}
for some constant $C>0$ depending on $k$, but independent of $n,\mcE$ and ${\bf c}$. 
\end{proposition}
\begin{proof}
By Lemma \ref{lemma47}-(b), one can directly  verify that 
\[u^{(n)}_{\bf c}=-c_4u_x^{(n)}+c_1u_y^{(n)}+(c_2+c_4)u_d^{(n)}
\]
satisfies \eqref{e:4.9}. Moreover, the estimate \eqref{e:4.10} follows from Lemma \ref{lemma45}. 
\end{proof}

\subsection{Proof of Proposition \ref{prop41}-(b)}\label{sec4.2}
In this subsection, we prove Proposition \ref{prop41}-(b). 

\medskip 

First, we introduce a sequence of approximating graphs associated with $F$. 
Recall that $q_5,q_6,q_7,q_8$ are the midpoints of $L_1,L_2,L_3,L_4$, respectively. Let $q_c=(1/2,1/2)$ be the center of $F_0$. Define the \textit{level-$0$ graph} $(\wt V_0, \wt E_0)$ by
\begin{align*} 
\wt{V}_0:=\{q_5,q_6,q_7,q_8,q_c\}\ \hbox{ and }\ \wt{E}_0:=\big\{\{q_i,q_c\}:i=5,6,7,8\big\},
\end{align*} 
and for $n\geq 1$, define the \textit{level-$n$ graph} $(\wt V_n, \wt E_n)$ by
\begin{align*}
\wt{V}_n:=\bigcup_{w\in W_n}\Psi_w(\wt{V}_0)\ \hbox{ and }\ \wt{E}_n:=\big\{\{\Psi_wq_i,\Psi_wq_c\}:\,w\in W_n,\,i=5,6,7,8\big\}. 
\end{align*}
We define a \textit{level-$n$ graph energy} $\wt\mcD_n$ on $(\wt V_n, \wt E_n)$ by 
\begin{align*}
\wt{\mcD}_n(f,g):=\sum_{\{p,q\}\in \wt{E}_n}\big(f(p)-f(q)\big)\big(g(p)-g(q)\big)\quad\hbox{ for }f,g\in l(\wt{V}_n). 
\end{align*}
Let $\wt{R}_n$ denote the effective resistance between subsets $A,B\subset\wt{V}_n$, i.e. 
\[
\wt{R}_n(A,B):=\left(\inf\Big\{\wt{\mcD}_n(f):\,f\in l(\wt{V}_n),\,f|_A=0\hbox{ and }f|_B=1\Big\}\right)^{-1}. 
\]
For simplicity, we write $\wt{R}_n(A,B)=\wt{R}_n(A\cap \wt{V}_n,B\cap\wt{V}_n)$ for $A,B\subset F$. 

Recall that $R_B$ is the effective resistance associated with $(\mcE_B,\mcF)$.

\begin{lemma}\label{lemma49}
There exists a constant $C>0$ such that 
\[
\wt{R}_n(F^A,F^B)\leq C\cdot3^{n(d_w-d_f)}R_B(F^A,F^B)
\]
for $n\geq 1$ and $A,B\subset W_n$ with $F^A\cap F^B=\emptyset$. 
\end{lemma}
\begin{proof}
Let $f\in l(\wt{V}_n)$ be the unique function satisfying 
\[
f|_{F^A\cap\wt{V}_n}=0,\ f|_{F^B\cap \wt{V}_n}=1\ \hbox{ and }\ \wt\mcD_n(f)=1/\wt{R}_n(F^A,F^B).
\]
Define $g\in l(V_n)$ by 
\[
g|_{F^A\cap V_n}=0,\quad g|_{F^B\cap V_n}=1,\]
and
\[g(q)=\frac{ \sum_{\{p\in\wt{V}_n:\,d(p,q)=3^{-n}/2\}}f(p)}{\#\{p\in\wt{V}_n:\,d(p,q)=3^{-n}/2\}}\quad\hbox{ for } q\in V_n\setminus{(F^A\cup F^B)}.
\]
It is straightforward to verify that for any $w\in W_n\setminus (A\cup B)$,  
\[
\sum_{i=1}^4\big(g(\Psi_wq_i)-\frac14\sum_{j=1}^4g(\Psi_wq_j)\big)^2\leq C_1\sum_{v\in W_n\atop F^v\cap F^w\neq\emptyset}\sum_{i=1}^4\big(f(\Psi_vq_{i+4})-f(\Psi_vq_c)\big)^2
\]
for some  constant $C_1>0$ independent of $f$, $w$ and $n$. Summing over all  $w\in W_n$, we obtain
\[
\sum_{w\in W_n}\sum_{i=1}^4\big(g(\Psi_wq_i)-\frac14\sum_{j=1}^4g(\Psi_wq_j)\big)^2\leq 8C_1\wt\mcD_n(f). 
\]
On the other hand, applying \eqref{e:prop4.1.1} to $\mcE^{(n)}_B$, we can construct a function $\wt g\in \mcF$ such that 
\[\wt g|_{F^A}=0,\quad \wt g|_{F^B}=1,
\]
and
\[
\mcE_B(\wt g)\leq \frac{C_2}{R_B^{(n)}(L_2,L_4)}\sum_{w\in W_n}\sum_{i=1}^4\big(g(\Psi_wq_i)-\frac14\sum_{j=1}^4g(\Psi_wq_j)\big)^2
\]
for some constant $C_2>0$ independent of $n$.
Combining this with the above estimate,
we have 
\[
\mcE_B(\wt g)\leq C_3\wt\mcD_n(f)/R^{(n)}_B(L_2,L_4)=C_43^{n(d_w-d_f)}\wt\mcD_n(f)
\]
for some constants $C_3,C_4>0$ independent of $A$, $B$ and $n$. Since $\mcE_B(\wt g)\geq 1/R_B(F^A,F^B)$ and $\wt \mcD_n(f)=1/\wt{R}_n(F^A,F^B)$, the desired inequality follows. 
\end{proof}

Next, we combine Proposition \ref{lemma48} and Lemma \ref{lemma49} to prove the following lemma. 

\begin{lemma}\label{lemma410}
Let $n\geq 1$ and $A,B\subset W_{n+k}$ with $F^A\cap F^B=\emptyset$. There exists a function $f\in L^2(F;\mu)$ such that 
\begin{eqnarray}
&f|_{F^w}\in\mcF^w\quad\hbox{ for all }w\in W_{n+k},\label{e:4.11}\\
&\mcE^{W_{n+k}}{(f,g)}=1\quad\hbox{ for all }g\in\mcF\hbox{ with }g|_{F^A}=0\hbox{ and }g|_{F^B}=1,\label{e:4.12}\\
&\mcE^{W_{n+k}}{(f)}\leq C\cdot 3^{n(d_w-d_f)}R_B(F^A,F^B)R^{(n)}(S_1,S_2),\label{e:4.13}
\end{eqnarray}
for some constant $C>0$ that depends on $k$ but is independent of $n$, $\mcE$, and $f$. 
\end{lemma}
\begin{proof}
We construct the function $f$ through the following three steps: 
\begin{enumerate}
\item[Step 1.] Choose $\wt{f}_1\in l(\wt{V}_{n+k})$ such that 
\[
\wt{f}_1|_{F^A\cap \wt{V}_{n+k}}=0,\ \wt{f}_1|_{F^B\cap \wt{V}_{n+k}}=1\ \hbox{ and }\  \wt\mcD_{n+k}(\wt{f}_1)=1/\wt{R}_{n+k}(F^A,F^B).
\]
Then define
\[
\wt{f}=\wt{f}_1/\wt\mcD_{n+k}(\wt{f}_1).
\] 
\item[Step 2.] For each $w\in W_{n+k}$, we assign a vector ${\bf c}_w=(c_{w,1},c_{w,2},c_{w,3},c_{w,4})\in \R^4$ by 
\[
c_{w,i}=\wt{f}(\Psi_wq_{i+4})-\wt{f}(\Psi_wq_c)\quad\hbox{ for }i=1,2,3,4.
\]
Note that $\sum_{i=1}^4c_{w,i}=0$ since $\wt f$ is $\wt\mcD_{n+k}$-harmonic at $\Psi_wq_c$.

\item[Step 3.] Apply Proposition \ref{lemma48} to define $f\in L^2(F;\mu)$ such that \eqref{e:4.11} holds and 
\begin{equation}\label{e:4.14}
\mcE^{(n+k)}(f\circ\Psi_w,g)=\sum_{i=1}^4c_{w,i}\Theta_i(g|_{L_i})\quad\hbox{ for all } w\in W_{n+k} \hbox{ and }g\in\mcF. 
\end{equation}
\end{enumerate}

First, we prove \eqref{e:4.13}. Note that 
\[
\sum_{w\in W_{n+k}}\sum_{i=1}^4c_{w,i}^2=\wt{\mcD}_{n+k}(\wt{f})=\wt{R}_{n+k}(F^A,F^B)\leq C_1\cdot 3^{n(d_w-d_f)}R_B(F^A,F^B),
\]
where the inequality follows from Lemma \ref{lemma49}, and $C_1>0$ is a constant depending on $k$ but independent of $n$.
By  Propositions \ref{prop32}-(c) and \ref{lemma48}, we then have
\[
\begin{split} 
\mcE^{W_{n+k}}({f})=\sum_{w\in W_{n+k}}\mcE^{(n+k)}(f\circ\Psi_w)&\leq C_2R^{(n)}(S_1,S_2)\sum_{w\in W_{n+k}}\sum_{i=1}^4c_{w,i}^2\\
&\leq C_33^{n(d_w-d_f)}R^{(n)}(S_1,S_2)R_B(F^A,F^B),
\end{split}
\]
for constants $C_2, C_3>0$ depending on $k$ but independent of $n$.
Thus,  \eqref{e:4.13} holds.

Second, we prove \eqref{e:4.12}. For any $\wt{g}\in l(\wt{V}_{n+k})$ with $\wt{g}|_{F^A\cap\wt{V}_{n+k}}=0$ and $\wt{g}|_{F^B\cap\wt{V}_{n+k}}=1$, we have  
\begin{equation}\label{e:4.15}
\begin{split}
1&=\wt{\mcD}_{n+k}(\wt{f},\wt{f}_1)=\wt{\mcD}_{n+k}(\wt{f},\wt{g})
=\sum_{w\in W_{n+k}}\sum_{i=1}^4c_{w,i}\big(\wt{g}(\Psi_wq_{i+4})-\wt{g}(\Psi_wq_c)\big) \\
& =\sum_{w\in W_{n+k}}\sum_{i=1}^4c_{w,i}\wt{g}(\Psi_wq_{i+4}),
\end{split}
\end{equation}
where the second equality holds because $\wt{f}$ is $\wt\mcD_{n+k}$-harmonic in $\wt{V}_{n+k}\setminus (F^A\cup F^B)$ and $\wt{g}=\wt{f}_1$ on $(F^A\cup F^B)\cap\wt{V}_{n+k}$, and the last equality follows from  $\sum_{i=1}^4c_{w,i}=0$. Now, for any $g\in\mcF$ such that $g|_{F^A}=0$ and  $g|_{F^B}=1$, define $\wt{g}\in l(\wt{V}_{n+k})$ by 
\[
\wt{g}(\Psi_wq_{i+4})=\Theta_i(g\circ\Psi_w|_{L_i}) \quad\hbox{ for }w\in W_{n+k}\hbox{ and }i=1,2,3,4, 
\]
and 
\[
\wt{g}(\Psi_wq_c)=\frac14\cdot {\sum_{i=1}^4 \wt g (\Psi_w q_{i+4})}\quad \hbox{ for }w\in W_{n+k}. 
\]
By Lemma \ref{lemma47}-(a), $\wt{g}$ is well-defined, and by Lemma \ref{lemma47}-(c), $\wt g|_{F^A\cap\wt V_{n+k}}=0$, $\wt g|_{F^B\cap\wt V_{n+k}}=1$.  Then,  
\begin{align*}
\mcE^{W_{n+k}}({f,g})&=\sum_{w\in W_{n+k}}\mcE^{(n+k)}(f\circ\Psi_w,g\circ\Psi_w)\\
&=\sum_{w\in W_{n+k}}\sum_{i=1}^4c_{w,i}\Theta_i(g\circ\Psi_w|_{L_i})\\
&=\sum_{w\in W_{n+k}}\sum_{i=1}^4c_{w,i}\wt{g}(\Psi_wq_{i+4})=1,
\end{align*}
where the first equality is due to Proposition \ref{prop32}-(c), the second uses \eqref{e:4.14}, and the last is due to \eqref{e:4.15}. This establishes \eqref{e:4.12}.
\end{proof}

\begin{corollary}\label{coro411}
Let $n\geq 1$ and $A,B\subset W_{n+k}$ with $F^A\cap F^B=\emptyset$. Then
\[
\mcE(g)\geq \frac{C^{-1}\cdot 3^{n(d_f-d_w)}}{R^{(n)}(S_1,S_2)R_B(F^A,F^B)}
\]
for any $g\in\mcF$ such that $g|_{F^A}= 0$ and $g|_{F^B}= 1$, where $C$ is the same constant as in Lemma \ref{lemma410}.
\end{corollary}
\begin{proof}
Let $f$ be the function constructed in Lemma \ref{lemma410}. By \eqref{e:4.12} and the Cauchy-Schwarz inequality,
\begin{align*}
\sqrt{\mcE^{W_{n+k}}{(f)}}\sqrt{\mcE( g)}
=\sqrt{\mcE^{W_{n+k}}{(f)}}\sqrt{\mcE^{W_{n+k}}{(g)}}\geq\mcE^{W_{n+k}}{(f,g)}=1.
\end{align*}
The desired estimate then follows from \eqref{e:4.13}. 
\end{proof}

By taking a limit in Corollary \ref{coro411}, we obtain an analogous inequality for any disjoint closed $A,B\subset F$. 
\begin{lemma}\label{lemma412}
Let $A,B$ be two disjoint closed subsets of $F$. Then
\[
\mcE(g)\geq C^{-1}\cdot\limsup_{n\to\infty}\frac{3^{n(d_f-d_w)}}{R^{(n)}(S_1,S_2)}\cdot \frac1{R_B(A,B)}
\]
for any $g\in\mcF\cap C(F)$ such that $g|_{A}\leq 0$ and $g|_{B}\geq 1$, where $C$ is the same constant as in Lemma \ref{lemma410}.
\end{lemma}
\begin{proof}
For sufficiently large $n\in\mathbb{N}$, say $n\geq n_0$, we choose $\{A_n,B_n\}$ to be the unions of $(n+k)$-cells such that  $A_{n_0}\supset A_{n_0+1}\supset A_{n_0+2}\supset\cdots$, $B_{n_0}\supset B_{n_0+1}\supset B_{n_0+2}\supset\cdots$,  with $A=\bigcap_{n\geq n_0}A_n$, $B=\bigcap_{n\geq n_0}B_n$, and $A_{n_0}\cap B_{n_0}=\emptyset$. 

Given any small $\delta>0$, the continuity of $g$ implies that there exists $N\geq n_0$ such that for all $n\geq N$, 
\begin{align*}
g|_{A_n}\leq \delta\ \hbox{ and }\ g|_{B_n}\geq 1-\delta.
\end{align*}
Define
\[
h=(\frac{g-\delta}{1-2\delta}\vee 0)\wedge 1,
\]
so that $h|_{{A_n}}=0$, $h|_{{B_n}}=1$ and $\mcE(h)\leq (1-2\delta)^{-2}\mcE(g)$. Then by Corollary \ref{coro411}, 
\[
(1-2\delta)^{-2}\mcE(g)\geq \mcE(h)\geq C^{-1}\frac{3^{n(d_f-d_w)}}{R^{(n)}(S_1,S_2)}\cdot \frac1{R_B(A_n,B_n)}.
\]
It follows that
\[
\mcE(g)\geq (1-2\delta)^2C^{-1}\limsup_{n\to\infty} \frac{3^{n(d_f-d_w)}}{R^{(n)}(S_1,S_2)}\cdot \frac1{R_B(A_n,B_n)}.
\]
Since $\delta$ is arbitrary, we conclude that
\[
\mcE(g)\geq C^{-1}\limsup_{n\to\infty} \frac{3^{n(d_f-d_w)}}{R^{(n)}(S_1,S_2)}\cdot \frac1{R_B(A_n,B_n)}.
\] 

It remains to prove that $\lim\limits_{n\to\infty}R_B({A_n},{B_n})=R_B(A,B)$. Let $h_n\in\mcF$ be such that 
\[
h_n|_{A_n}=0,\ h_n|_{B_n}=1,\hbox{ and } \mcE_B(h_n)=1/R_B(A_n,B_n).
\]
There exists a subsequence $h_{n_k}$  that weakly converges to some $h\in \mcF$ in the sense that 
\[
\mcE_B(h_{n_k},g)\to \mcE_B(h,g)\ \hbox{ and }\ \int_F h_{n_k}gd\mu\to \int_F hgd\mu\ \hbox{ as }k\to\infty\quad\hbox{ for all } g\in\mcF. 
\]
In particular, $h|_A=0$, $h|_B=1$, and $h$ is $\mcE_B$-harmonic in $F\setminus(A\cup B)$. Consequently, $\mcE_B(h)=1/R_B(A,B)$. Moreover, by Mazur's lemma, a subsequence of the Ces\`aro means $\frac1L\sum_{l=1}^{L}h_{n_l}$ converges strongly in $\mcF$ to $h$ as $L\to\infty$. Note that
\begin{align*}
\mcE_B(\frac1L\sum_{l=1}^{L}h_{n_l})&=\frac1{L^2}\Big(\sum_{l=1}^L\mcE_B(h_{n_l})+2\sum_{l=2}^L\sum_{j=1}^{l-1}\mcE_B(h_{n_l},h_{n_j})\Big)\\
&=\sum_{l=1}^L\frac{2l-1}{L^2}\mcE_B(h_{n_l})=\sum_{l=1}^L\frac{2l-1}{L^2}\frac{1}{R_B(A_{n_l},B_{n_l})},
\end{align*} 
where we use the fact that $\mcE_B(h_{n_l}, h_{n_j})=\mcE_B(h_{n_l})$ for $j<l$ due to the harmonicity and boundary values. Therefore,
\[
\frac1{R_B(A,B)}=\mcE_B(h)=\lim_{L\to\infty}\sum_{l=1}^L\frac{2l-1}{L^2}\frac{1}{R_B(A_{n_l},B_{n_l})}=\lim_{n\to\infty}\frac1{R_B(A_n,B_n)},
\] 
where the last equality follows from the monotonicity of $R_B(A_n,B_n)$.
\end{proof}

With Lemma \ref{lemma412} in hand, we now prove Proposition \ref{prop41}-(b).

\begin{proof}[Proof of Proposition \ref{prop41}-(b)]
Without loss of generality, we may assume 
\[
\min_F f=0\hbox{ and }\max_F f=1,
\]
by replacing $f$ with $(f-\inf_F f)/(\sup_F f-\inf_F f)$ if necessary. For $L\geq 2$, we rewrite $f$ as 
\[
f=\sum_{l=1}^L f_l\ \hbox{ with }f_l=\big(0\vee(f-\frac{l-1}{L})\big)\wedge \frac1L\quad \hbox{ for }1\leq l\leq L. 
\]
Let $g_{L,l}\in\mcF$ be such that 
\begin{align*} 
g_{L,l}=0\hbox{ on }\{f\leq \frac{l-1}L\},\quad 
g_{L,l}=\frac1L\hbox{ on }\{f\geq \frac{l}L\},\\
\hbox{ and }\quad g_{L,l}\hbox{ is }\mcE_B\hbox{-harmonic in }F\setminus \big\{f\leq \frac{l-1}L\hbox{ or }f\geq \frac{l}L\big\}. 
\end{align*}
Then, by Lemma \ref{lemma412}, 
\[
\mcE(f_l)\geq C^{-1}\limsup_{n\to\infty}\frac{3^{n(d_f-d_w)}}{R^{(n)}(S_1,S_2)}\mcE_B(g_{L,l}),
\]
where $C$ is the same constant as in Lemma \ref{lemma410}.
By the strong local property, we have 
\begin{align*} 
\mcE(f)=\sum_{l=1}^L\mcE(f_l)&\geq C^{-1}\limsup_{n\to\infty}\frac{3^{n(d_f-d_w)}}{R^{(n)}(S_1,S_2)}\sum_{l=1}^L\mcE_B(g_{L,l})\\
&=C^{-1}\limsup_{n\to\infty}\frac{3^{n(d_f-d_w)}}{R^{(n)}(S_1,S_2)}\mcE_B(\sum_{l=1}^L g_{L,l}). 
\end{align*}
Since $\sum_{l=1}^L g_{L,l}$ converges uniformly to $f$ as $L\to\infty$, the desired inequality follows from the lower-semicontinuity of $\mcE_B$.
\end{proof}

\section{Uniqueness}
In this section, we prove Theorem \ref{thm2}. 
The main step is to show
  that there exists a constant $C\in(1,\infty)$, independent of $(\mcE,\mcF)$, such that 
\begin{equation}\label{eqn51}
\frac{\sup(\mcE|\mcE_B)}{\inf(\mcE|\mcE_B)}\leq C. 
\end{equation}
 We use a proof by contradiction argument.

\smallskip

It is known from \cite[Remark 5.4]{BB3} that $3^{d_w-d_f}\leq 3/2$. In the next lemma, we show $3^{d_w-d_f}<3/2$. This strict inequality is needed later in our proof of the uniqueness theorem, Theorem \ref{thm2}.

\begin{lemma}\label{lemma51}
$3^{d_w-d_f}<3/2$, that is $d_w<1+\frac{\log4}{\log3}$. 
\end{lemma}

\begin{proof}
Recall that for any $n\geq 0$ and $f\in\mcF$ (see Remark \ref{remark33}), 
\[
\mcE^{(n)}_B(f)=3^{n(d_w-d_f)}\mcE_B(f). 
\]
We focus on the cases $n=0,1$. Let $h\in\mcF$ satisfy 
\[
h|_{L_4}=0,\ h|_{L_2}=1,\hbox{ and }\mcE_B(h)=1/{R_B(L_2,L_4)}. 
\]
Define $A=\{1,2,3,4,5,7\}$, which corresponds to two rows of $1$-cells attached to $L_1$ and $L_3$, and define $h_1\in C(F^A)$ by 
\[
h_1=\frac13\Big(h\circ\Psi_1^{-1}+h\circ\Psi_4^{-1}+(h+1)\circ\Psi_5^{-1}+(h+1)\circ\Psi_7^{-1}
+(h+2)\circ\Psi_2^{-1}+(h+2)\circ\Psi_3^{-1}\Big).
\]
We can verify that (by Proposition \ref{prop32}-(c) and Remark \ref{remark33})
\begin{align*}
&\sum_{i\in A}\mcE^i_B\big((h-h_1)|_{F^i},h_1|_{F^i}\big)\\=&3^{d_w-d_f}\sum_{i\in A}\mcE_B\big((h-h_1)\circ\Psi_i,h_1\circ\Psi_i\big)\\
=&3^{d_w-d_f-1}\sum_{i\in A}\mcE_B\big((h-h_1)\circ\Psi_i,h\big)\\
=&3^{d_w-d_f-1}\Big(\sum_{i=1,2,3,4}\mcE_B\big((h-h_1)\circ\Psi_i,h\big)-\sum_{i=5,7}\mcE_B\big((h-h_1)\circ\Psi_i\circ\mcG_x,h\big)\Big)\\
=&3^{d_w-d_f-1}\mcE_B\Big(\sum_{i=1,2,3,4}(h-h_1)\circ\Psi_i-\sum_{i=5,7}(h-h_1)\circ\Psi_i\circ\mcG_x,h\Big)\\
=&0,
\end{align*}
where the third equality holds because  $h\circ\mcG_x=-h+1$, and the last equality holds since $\sum_{i=1,2,3,4}(h-h_1)\circ\Psi_i-\sum_{i=5,7}(h-h_1)\circ\Psi_i\circ\mcG_x$ vanishes on $L_2\cup L_4$, and $h$ is $\mcE_B$-harmonic in $F\setminus (L_2\cup L_4)$. 
This  implies, in particular (by Lemma \ref{lemma31}), 
\[
\mcE_B(h)-6\cdot 3^{d_w-d_f-2}\mcE_B(h)=\mcE_B(h)-\sum_{i\in A}\mcE_B^i(h_1|_{F^i})=\sum_{i=6,8}\mcE_B^i(h|_{F^i})+\sum_{i\in A}\mcE_B^i\big((h-h_1)|_{F^i}\big)>0,
\]
where the inequality can be seen by considering two cases: if $h-h_1\neq 0$ on $\bigcup_{i\in A}F^i$, then $\sum_{i\in A}\mcE^i_B\big((h-h_1)|_{F^i}\big)>0$; if $h-h_1= 0$, then $h|_{F^6\cup F^8}$ cannot be constant, so $\sum_{i=6,8}\mcE^i_B(h|_{F^i})>0$.   The claim follows immediately from the inequality. 
\end{proof}

\begin{lemma}\label{lemma52}
Let $h$ be the function satisfying 
\[
h|_{L_4}=0,\ h|_{L_2}=1,\hbox{ and }h\hbox{ is }\mcE\hbox{-harmonic in }F\setminus(L_2\cup L_4).
\]
Then for $m\geq 1$,
\[
\mu_{\<h\>}\big(F\cap [0,1]\times [0,3^{-m}]\big)\leq 2^{-m}\mcE(h). 
\]
\end{lemma}
\begin{proof}
The proof follows along the same line as that for \cite[Lemma 2.6]{BHHW}. For the reader's convenience, we spell out the details.

For $m\geq 1$, we define $h_m$ by 
\[
h_m(x,y):=\begin{cases}
h(x,y)&\hbox{ if }(x,y)\in F\cap [0,1]\times[3^{-m}/2,1],\\
h(x,3^{-m}-y)&\hbox{ if }(x,y)\in F\cap [0,1]\times[0,3^{-m}/2).
\end{cases}
\]
It is easy to see that $h_m\in\mcF$, by using the Besov characterization of $\mcF$ \cite[Theorem 4.2]{GHL}. Note that $h_m|_{L_4}=0$ and $h_m|_{L_2}=1$ still hold, so $\mcE(h_m)\geq \mcE(h)$.  Together with the fact that $\mu_{\<h\>}\big(F\cap[0,1]\times[3^{-m}/2,1]\big)=\mu_{\<h_m\>}\big({ F\cap[0,1]\times[3^{-m}/2,1]}\big)$, we obtain 
\begin{align*} 
\mu_{\<h\>}\big(F\cap [0,1]\times [0,3^{-m-1}]\big)&\leq \mu_{\<h\>}\big(F\cap [0,1]\times [0,3^{-m}/2)\big)\\
&\leq \mu_{\<h_m\>}\big(F\cap [0,1]\times [0,3^{-m}/2)\big)\\
&= \mu_{\<h\>}\big(F\cap [0,1]\times (3^{-m}/2,3^{-m}]\big),
\end{align*}
where the last equality uses the reflection symmetry of $\mcE^{w}$ for each $w\in W_m$.
Thus
\begin{align*}
\mu_{\<h\>}\big(F\cap [0,1]\times [0,3^{-m-1}]\big)\leq \frac12\mu_{\<h\>}\big(F\cap [0,1]\times [0,3^{-m}]\big).
\end{align*} 
The desired conclusion then follows by iteration. 
\end{proof}

\begin{lemma}\label{lemma53}
For $k\geq 2$, with $R^{(n)}_x,R^{(n)}_y, R^{(n)}(S_1,S_2)$ defined as in Section \ref{sec4}, it holds that 
\begin{align}
\label{e:lemma5.3.1}
\frac1{R^{(n)}_x}\leq\frac{6+2^{-k+3}}{9R^{(n+1)}_x}+\frac4{R^{(n)}(S_1,S_2)}\quad\hbox{ for }n\geq 0,\\
\label{e:lemma5.3.2}\frac1{R^{(n)}_y}\leq \frac{6+2^{-k+3}}{9R^{(n+1)}_y}+\frac4{R^{(n)}(S_1,S_2)}\quad\hbox{ for }n\geq 0.
\end{align}
\end{lemma}
\begin{proof}
We prove only \eqref{e:lemma5.3.1}, as inequality \eqref{e:lemma5.3.2} follows by the same argument. 

Let $h\in\mcF$ satisfy 
\[
h|_{L_4}=0,\ h|_{L_2}=1,\hbox{ and $h$ is }\mcE^{(n+1)}\hbox{-harmonic in }F\setminus(L_2\cup L_4). 
\]
Let $g_0\in\mcF$ satisfy 
\[
g_0|_{S_1\cup \mcG_y(S_1)}=0,\ g_0|_{S_2\cup \mcG_y(S_2)}=1,\hbox{ and  $g_0$ is }\mcE^{(n)}\hbox{-harmonic in }F\setminus\big(S_1\cup \mcG_y(S_1)\cup S_2\cup\mcG_y(S_2)\big);
\]
and define 
\[
g_1=(1-g_0)\circ\mcG_x. 
\]
Clearly, 
\[
\mcE^{(n+1)}(h)=1/R^{(n+1)}_x\hbox{ and }\mcE^{(n)}(g_0)=\mcE^{(n)}(g_1)\leq2/R^{(n)}(S_1,S_2). 
\]

Now, we construct $f\in\mcF$ in two steps. 
\begin{enumerate}
\item Define 
\begin{align*}
\wh {f}=&\frac13h\circ\Psi_1^{-1}+\frac13h\circ\Psi_8^{-1}+\frac13 h\circ\Psi_4^{-1}\\
&+\frac13(h+1)\circ\Psi_5^{-1}+\frac13(h+1)\circ\Psi_7^{-1}\\
&+\frac13(h+2)\circ\Psi_2^{-1}+\frac13(h+2)\circ\Psi_6^{-1}+\frac13(h+2)\circ\Psi_3^{-1}. 
\end{align*}

\item Set $f=(\wh {f}\wedge g_0)\vee g_1$.
\end{enumerate}
Then $f$ satisfies $f|_{L_4}=0$, $f|_{L_2}=1$, and takes the constant value $0$ on $F\cap [0,\frac13]\times[\frac13+3^{-k},\frac23-3^{-k}]$ and the constant value $1$ on $F\cap [\frac23,1]\times[\frac13+3^{-k},\frac23-3^{-k}]$. Applying Lemma \ref{lemma52} together with Lemmas \ref{lemma31} and \ref{lemma35}-(a), yields
\[
\mu^{(n)}_{\<\wh {f}\>}\big(F\setminus[0,1]\times[\frac13+3^{-k},\frac23-3^{-k}]\big)\leq \frac19(6+4\cdot 2^{-(k-1)})\mcE^{(n+1)}(h)=\frac19(6+2^{-k+3})\frac1{R_x^{(n+1)}}. 
\]
Finally, we obtain 
\begin{align*}
\frac1{R^{(n)}_x}\leq\mcE^{(n)}(f)&=\mu^{(n)}_{\<f\>}\big(F\setminus[0,1]\times[\frac13+3^{-k},\frac23-3^{-k}]\big)\\
\\&\leq \mu^{(n)}_{\<\wh {f}\>}\big(F\setminus[0,1]\times[\frac13+3^{-k},\frac23-3^{-k}]\big)+\mcE^{(n)}(g_0)+\mcE^{(n)}(g_1)\\
\\&\leq \frac19(6+ 2^{-k+3})\frac1{R^{(n+1)}_x}+\frac4{R^{(n)}(S_1,S_2)}.
\end{align*}
This establishes \eqref{e:lemma5.3.1}.
\end{proof}

We are now in the position to prove our uniqueness theorem.

\begin{proof}[Proof of Theorem \ref{thm2}]
We first prove \eqref{eqn51} for each 
symmetric strongly local regular Dirichlet form  $(\mcE,\mcF)$ on $L^2(F; \mu)$
satisfying the assumptions of the theorem. By Lemma \ref{lemma51}, we have $1-\frac23\cdot 3^{d_w-d_f}>0$. Choose a  sufficiently small $\delta>0$ (with $\delta<1-\frac23\cdot 3^{d_w-d_f}$)  and fix a sufficiently large  $k$ such that 
\begin{equation}\label{e:5.4} 
\frac19(6+2^{-k+3})3^{d_w-d_f}\leq 1-\delta. 
\end{equation} 
Take a subsequence $n_l,l\geq 0$ for which 
\begin{equation*} 
\frac{3^{n_l(d_f-d_w)}}{R^{(n_l)}_x\wedge R_y^{(n_l)}}\to \liminf_{n\to\infty}\frac{3^{n(d_f-d_w)}}{R_x^{(n)}\wedge R_y^{(n)}}\quad \hbox{ as }l\to\infty,
\end{equation*} 
noting that by Remarks \ref{remark42} and \ref{remark24} the latter limit is in $(0,\infty)$. 
Consequently,
\[
3^{d_f-d_w}\frac{R^{(n_l-1)}_x\wedge R_y^{(n_l-1)}}{R^{(n_l)}_x\wedge R_y^{(n_l)}}\leq 1+\delta \quad\hbox{ for all sufficiently large } l\geq 1.  
\]
Using inequalities \eqref{e:lemma5.3.1}, \eqref{e:lemma5.3.2} together with \eqref{e:5.4}, we then obtain
\begin{align}\label{e:5.6}
1&\leq \frac19(6+2^{-k+3})\frac{R_x^{(n_l-1)}\wedge R_y^{(n_l-1)}}{R_x^{(n_l)}\wedge R_y^{(n_l)}}+4\frac{{R_x^{(n_l-1)}}\wedge R_y^{(n_l-1)}}{{ R^{(n_l-1)}(S_1,S_2)}}\nonumber\\
&\leq (1-\delta^2)+4\frac{R^{(n_l-1)}_x\wedge R_y^{(n_l-1)}}{R^{(n_l-1)}(S_1,S_2)}\quad\hbox{ for large } l \geq 1 .
\end{align}
Hence, 
\begin{align*}
\limsup_{n\to\infty}\frac{3^{n(d_f-d_w)}}{R^{(n)}(S_1,S_2)}&\geq \limsup_{l\to\infty}\frac{3^{(n_l-1)(d_f-d_w)}}{R^{(n_l-1)}(S_1,S_2)}\\
&\geq \frac{\delta^2}{4}\limsup_{l\to\infty}\frac{3^{(n_l-1)(d_f-d_w)}}{R^{(n_l-1)}_x\wedge R_y^{(n_l-1)}}\geq \frac{\delta^2}{4}\liminf_{n\to\infty}\frac{3^{n(d_f-d_w)}}{R^{(n)}_x\wedge R_y^{(n)}},
\end{align*}
where the second inequality follows from \eqref{e:5.6}.
Finally, \eqref{eqn51} follows from Proposition \ref{prop41}. 

From  \eqref{eqn51} we deduce  that for any two symmetric strongly local regular Dirichlet forms $(\mcE,\mcF)$ and $(\mcE',\mcF)$ on $L^2(F; \mu)$
satisfying the assumptions of the theorem,
\begin{equation}\label{e:5.7} 
\frac{\sup(\mcE'|\mcE)}{\inf(\mcE'|\mcE)}\leq C^2.
\end{equation}

The rest of the proof proceeds by contradiction, following the  same argument as that for \cite[Theorem 1.2]{BBKT}. Assume that there exist $(\mcE',\mcF)$ and $(\mcE,\mcF)$ which are not constant multiples of each other, i.e. 
\[
\frac{\sup(\mcE'|\mcE)}{\inf(\mcE'|\mcE)}>1. 
\]
For $\eta>0$, define 
\[\mcE''=(1+\eta)\mcE'-\inf(\mcE'|\mcE)\mcE.
\]
By \cite[Theorem 2.1]{BBKT}, $(\mcE'',\mcF)$ is a strongly local, regular Dirichlet form. Moreover, $\mcE''$ also satisfies the assumptions of the theorem: 
\begin{enumerate}
    \item ${\bf HK}({d_w})$ holds because $\eta\mcE'\leq \mcE''\leq \big((1+\eta)\sup(\mcE'|\mcE)-\inf(\mcE'|\mcE)\big)\mcE$, and the stability theorem for heat kernel estimates  \cite{AB,GHL2} applies. 
    
    \item the reflection symmetry holds since for any $f\in\mcF$,
    \begin{align*} 
    \mcE''(f\circ\mcG_x)&=(1+\eta)\mcE'(f\circ\mcG_x)-\inf(\mcE'|\mcE)\mcE(f\circ\mcG_x)\\
    &=(1+\eta)\mcE'(f)-\inf(\mcE'|\mcE)\mcE(f)=\mcE''(f),
    \end{align*}
    and similarly $\mcE''(f\circ\mcG_y)=\mcE''(f)$. 
    
    \item the local symmetry follows by an argument analogous to that for the reflection symmetry. 
\end{enumerate}
Now observe that
\[
\frac{\sup(\mcE''|\mcE)}{\inf(\mcE''|\mcE)}=\frac{(1+\eta)\sup(\mcE'|\mcE)-\inf(\mcE'|\mcE)}{\eta \inf(\mcE'|\mcE)}>\frac{1}{\eta}\cdot\big(\frac{\sup(\mcE'|\mcE)}{\inf(\mcE'|\mcE)}-1\big).
\]
For sufficiently small $\eta$, this contradicts \eqref{e:5.7}  with $\mcE'$ replaced by $\mcE''$.
\end{proof}

\section{Convergence}
In this section, we prove the \textit{strong homogenization}  theorem (Theorem \ref{thm1}) and the convergence of Dirichlet forms 
$\{ \big(3^{n(d_w-d_f)}\mcE^{(r)}_{F_n},W^{1,2}(F_n)\big); n\geq 1\}$ introduced in Section \ref{sec23} 
and their associated reflected diffusions
(Theorem \ref{thm3}). In what follows, we always abbreviate $\mcE_{F_n}^{(r)}$ to $\mcE_{F_n}$.

\subsection{Mosco convergence}
This section is devoted to the proof of Theorem \ref{thm3}. Since well-established frameworks already exist, concerning tightness of the processes in \cite{BB, BB3} and passage from sub-sequential limits to sequential limits in \cite{BBKT,CC}, we only briefly outline the main ideas here. 

Before proceeding, we introduce some notation. 
\begin{enumerate}
\item For a Hunt process $(Y_t)_{t\geq 0}$ on a locally compact metric measure space $(\mathcal M,d)$, and for a subset $A\subset \mathcal M$, we write 
\[
\tau_A=\inf\{t\geq0:\,Y_t\notin A\}\quad\hbox{ and }\sigma_A=\inf\{t\geq 0:\,Y_t\in A\}
\]
for the first exit time and hitting time, respectively. 

\item We denote by $(X^{(n)}_t)_{t\geq 0}$ the Hunt process associated with $\big(3^{n(d_w-d_f)}\mcE_{F_n},W^{1,2}(F_n)\big)$ on $L^2(F_n;\mu_n)$. We write $ \mathbb{P}^{(n)}_z$ for the probability distribution of the process starting at $z\in F_n$ and denote by 
$\mathbb{E}^{(n)}_z$ its corresponding expectation. 
    
\item We denote by $(X_t)_{t\geq 0}$ the Brownian motion on $F$ associated with $(\mcE_B,\mcF)$. We write $\mathbb{P}_z$ for the probability distribution of the process starting at $z\in F$ and denote by $\mathbb{E}_{z}$ the corresponding expectation.
\end{enumerate}

\begin{lemma}\label{lemma61}
\begin{enumerate}
	\item[(a)] The uniform elliptic Harnack inequality holds for $\big(\mcE_{F_n},W^{1,2}(F_n)\big)$, i.e. there exist positive constants $C>0$ and $c\in(0,1)$, independent of $n$ (but depending on $r$), such that 
	\[
	h(z)\leq C\cdot h(z')\quad\hbox{ for all } z, z'\in B(z_0, c\cdot\rho),
	\]
	for any $n\geq 0$, $z_0\in F_n$, $\rho\in(0,1/2)$, and any nonnegative function $h\in W^{1,2}(F_n)$ that is $\mcE_{F_n}$-harmonic in $B(z_0,\rho)$.
	\item[(b)]There is a positive constant $C'$, independent of $n$ (but depending on $r$), such that 
	\[
	C'^{-1}\cdot\big(3^{n(2-d_w)}\rho^2 \vee\rho^{d_w}\big)\leq \mathbb{E}^{(n)}_z [\tau_{B(z,\rho)}]\leq C'\cdot\big( 3^{n(2-d_w)}\rho^2 \vee\rho^{d_w}\big)
	\]
	for any $n\geq 0$, $z\in F_n$, $\rho\in(0,1/2)$.
	Moreover, there is a positive constant $C''>1$, independent of $n$ (but depending on $r$), such that 
	\[
	C''^{-1}\leq
	\mathbb{E}^{(n)}_{q_1} [\sigma_{L_2\cup L_3}]\leq C''. 
	\]
	{ Recall that $q_1=(0,0)$, $L_2$ and $L_3$ are the two boundary segments of $F_0$ defined at the beginning of Section 2.} 
\end{enumerate}
\end{lemma}

\begin{proof}
Let $\wh {F}_0=\bigcup_{n=0}^\infty 3^nF_n$ be the unbounded pre-carpet, and consider the divergence form $\big(\mcE_{\wh {F}_0},W^{1,2}(\wh {F}_0)\big)$ defined by
\[
\mcE_{\wh  F_0}(f,g)=\int_{\wh  F_0}\left(\Big(\frac{\partial f}{\partial x}\cdot\frac{\partial g}{\partial x}\Big)(x,y)+\frac1r\Big(\frac{\partial f}{\partial y}\cdot\frac{\partial g}{\partial y}\Big)(x,y)\right)dxdy.
\]
 Let  $(\wh {X}_t)_{t\geq 0}$ be the reflected Brownian motion on $\wh {F}_0$ associated with $\big(\mcE_{\wh  F_0}, W^{1,2}(\wh  F_0)\big)$. 
We write $\wh {\mathbb{P}}_z$ for the probability distribution of  $(\wh {X}_t)_{t\geq 0}$ starting at $z\in \wh {F}_0$, and denote by $\wh {\mathbb{E}}_z$ its corresponding expectation.

(a) It is known that for $r=1$, the elliptic Harnack inequality holds \cite[Theorem 1.1]{BB3}. For other values of $r\in(0,\infty)$, the stability theorem for the elliptic Harnack inequality \cite[Theorem 1.3]{BCM} implies that a scale-invariant Harnack inequality also holds for $\big(\mcE_{\wh {F}_0},W^{1,2}( \wh  F_0)\big)$ with the comparison constant independent of $n$ but dependent on $r>0$. In particular, the reflected Dirichlet forms $\big(\mcE_{3^nF_n},W^{1,2}(3^nF_n)\big)$ defined by
\[
\mcE_{3^n F_n}(f,g)=\int_{ 3^nF_n}\left(\Big(\frac{\partial f}{\partial x}\cdot\frac{\partial g}{\partial x}\Big)(x,y)+\frac1r\Big(\frac{\partial f}{\partial y}\cdot\frac{\partial g}{\partial y}\Big)(x,y)\right)dxdy
\]
satisfies the same scale-invariant Harnack inequality with the same constant as $\big(\mcE_{\wh {F}_0},W^{1,2}( \wh  F_0)\big)$. Indeed, the inequality clearly holds with the same constant on any ball that does not intersect $[0,3^n]\times\{3^n\}$ and $\{3^n\}\times [0,3^n]$; by symmetry, it then holds for any ball of radius $\rho\in (0,3^n/2)$. 

Finally, (a) follows because $\big(\mcE_{F_n},W^{1,2}(F_n)\big)$ is a scaled copy of $\big(\mcE_{3^nF_n},W^{1,2}(3^nF_n)\big)$. 
\medskip

(b) For $r=1$, the sub-Gaussian heat kernel estimate for $\big(
\mcE_{\wh  F_0}, W^{1,2}(\wh  F_0)\big)$ with exit time profile $\phi(\rho)=(\rho^{d_w}\vee \rho^2)$ holds by \cite[Theorem 1.4]{BB3}.   By the stability theorem for sub-Gaussian heat kernel estimates \cite{AB,GHL2}, it remains valid for  general $r\in(0,\infty)$, i.e.
\[
C_1^{-1}(\rho^{d_w}\vee \rho^2)\leq \wh {\mathbb{E}}_z[\tau_{B(z,\rho)}] 
\leq C_1(\rho^{d_w}\vee \rho^2)\quad\hbox{ for }z\in\wh {F}_0, \rho>0,
\]
for some constant $C_1>0$ depending on $r$.  

Using localization and symmetry as in (a), the reflected process $\wh {X}^{(n)}_t$ associated with $\big(\mcE_{3^nF_n},W^{1,2}(3^nF_n)\big)$ also satisfies 
\[
C_1^{-1}(\rho^{d_w}\vee \rho^2)\leq  \wh {\mathbb{E}}^{(n)}_z [\tau_{B(z,\rho)}]\leq C_1(\rho^{d_w}\vee \rho^2)\quad\hbox{ for }z\in 3^nF_n, \rho\in(0,3^n/2),
\]
 where  $\wh {\mathbb{E}}^{(n)}_z$ denotes the expectation with respect to the law of $\wh {X}^{(n)}_t$  starting from $z$.

The first claim of (b) then follows by scaling. The second claim follows from a similar argument. 
\end{proof}

\begin{proof}[Proof of Theorem \ref{thm3}]
The proof proceeds in several steps, all of which use the well-established approaches in previous works \cite{BB,BB3, BBKT, CC}. \medskip

Denote by $ U^{(n)}_\lambda$   the $\lambda$-resolvent operator associated with $(X^{(n)}_t)_{t\geq 0}$, i.e.
\[
 U^{(n)}_\lambda f(z)=\mathbb{E}_{z}^{(n)} \int_{0}^\infty e^{-\lambda t} f(X^{(n)}_t) dt\quad\hbox{ for }f\in C(F_n),\,z\in F_n. 
\]
Note that $ U^{(n)}_\lambda$ extends to a bounded operator on $L^2(F_n;\mu_n)$. Based on Lemma \ref{lemma61}, the same proof as in \cite[Section 3.1]{BBKT} shows that $ U^{(n)}_\lambda (f |_{F_n})$ is uniformly continuous for any bounded $f$ on $F_0$. Then, following the same arguments as in \cite[Sections 5 and 6]{BB}, for any subsequence there exists a further subsequence $\{n_l; l\geq 0\}$ and a Hunt process $(\wt{X}_t)_{t\geq 0}$ with probability law $\{ \wt{\mathbb{P}}_z; z\in F\}$ such that for any $z_{n_l}\in F_{n_l}$ and $z\in F$ with  $z_{n_l}\to z$,  the processes $(X^{(n_l)}_t)_{t\geq 0}$ under ${\mathbb{P}}^{(n_l)}_{z_{n_l}}$ converge weakly in $C ([0,\infty); \R^d)$ to $\wt{X}_t$ under $\wt{\mathbb{P}}_z$.\smallskip 

The weak convergence $(X^{(n_l)}_t)_{t\geq 0} \to (\wt{X}_t)_{t\geq 0}$ implies that 
\[
U^{(n_l)}_\lambda (f |_{F_{n_l}})(z_{n_l})\to \wt{U}_{\lambda}(f|_F)(z)
\]
 for any $z_{n_l}\in F_{n_l}$ and $z\in F$ with  $z_{n_l}\to z
$, and 
for any bounded $f$ on $F_0$, where $\wt{U}_\lambda$ is the resolvent operator associated with $(\wt{X}_t)_{t\geq 0}$.  Let $(\wt{\mcE},\wt{\mcF})$ be the Dirichlet form associated with $\wt{X}_t$. Then, by \cite[Lemma 2.11-(b), Theorem 2.13 and Proposition 2.15]{CC}, $3^{n_l(d_w-d_f)}\mcE_{F_{n_l}}$ is Mosco convergent to $\wt{\mcE}$. Moreover, $(\wt{\mcE},\wt{\mcF})$ satisfies ${\bf HK}({d_w})$, local symmetry, and reflection symmetry.
\begin{enumerate}
\item Since $\wt{\mcE}$ is the Mosco limit of $3^{n_l(d_w-d_f)}\mcE_{F_{n_l}}$, and $\mcE_B$ is the Mosco limit of the standard divergence forms {\eqref{e:0}} on $F_n$ \cite[Theorem 1.3]{CC}, it holds that 
\[
C_1^{-1}\mcE_B\leq \wt{\mcE}\leq C_1\mcE_B\ \hbox{ for some constant }C_1>1. 
\]
Therefore, ${\bf HK}(d_w)$ holds for $(\wt{\mcE},\wt{\mcF})$ by the stability theorem for heat kernel estimates \cite{AB,GHL2}.  Consequently, due to the Besov-type characterization of $\wt\mcF$ whenever {\bf HK}($d_w$) holds  \cite[Theorem 4.2]{GHL}, we have $\wt\mcF=\mcF$, the domain of the Dirichlet form $\mcE_B$.

\item For $f\in{\mcF}$, by Mosco convergence there exists a sequence  $f_{n_l}\in W^{1,2}(F_{n_l})$ that converges strongly to $f$ in $L^2$ such that 
\[
\wt{\mcE}(f)=\lim_{l\to\infty}3^{n_l(d_w-d_f)}\mcE_{F_{n_l}}(f_{n_l}).
\]
Hence,
\[
\wt{\mcE}(f\circ\mcG_x)\leq \lim_{l\to\infty}3^{n_l(d_w-d_f)}\mcE_{F_{n_l}}(f_{n_l}\circ\mcG_x)=\lim_{l\to\infty}3^{n_l(d_w-d_f)}\mcE_{F_{n_l}}(f_{n_l})=\wt{\mcE}(f). 
\]
By symmetry, it also holds that $\wt{\mcE}(f)\leq \wt{\mcE}(f\circ\mcG_x)$, so $\wt{\mcE}(f)=\wt{\mcE}(f\circ\mcG_x)$. By exactly the same proof, $\wt{\mcE}(f\circ\mcG_y)=\wt{\mcE}(f)$. 

\item The local symmetry follows from the same proof as in \cite[Theorem 3.1]{BBKT}. 
\end{enumerate}
As a consequence of Theorem \ref{thm2}, we know that $\wt{\mcE}=C_2\mcE_B$ for some positive constant $C_2$ depending on $r$ and the subsequence $\{n_l\}$. \medskip 

Next, set
\[
\alpha_n:=\mathbb{E}_{n,q_1}[\sigma_{L_2\cup L_3}],
\]
and consider the constant time changed process $\wt{X}^{(n)}_t:=X^{(n)}_{\alpha_nt}$. Denote by $\wt{\mathbb{P}}^{(n)}_z$ and 
$\wt{\mathbb{E}}^{(n)}_z$ the law and expectation associated with $\wt{X}^{(n)}_t$, respectively. Note  that 
\[
1=\wt{\mathbb{E}}^{(n)}_{q_1}  [\sigma_{L_2\cup L_3}],
\]
and $(\wt{X}^{(n)}_t)_{t\geq 0}$ is associated with the Dirichlet form $\big(\alpha_n3^{n(d_w-d_f)}\mcE_{F_n},W^{1,2}(F_n)\big)$. Then, by the previous step and following the same proof as in \cite[Remark 5.4]{BBKT} and a more detailed proof in \cite[Theorems 3.10 and 3.12]{CC}, we conclude that for some constant $C_3>0$, $\wt{X}^{(n_l)}_t$ weakly converges to $X_{C_3t}$ and $\alpha_{n_l}3^{n_l(d_w-d_f)}\mcE_{F_{n_l}}$ is Mosco convergent to $C_3\mcE_B$ as $l\to\infty$.\medskip

Finally, the limit $\lim\limits_{n\to\infty}\alpha_n$ exists by the same proof as that for \cite[Theorem 1.3]{CC}. Hence, $3^{n(d_w-d_f)}\mcE_{F_n}$ is Mosco convergent to $C\mcE_B$ for some $C>0$ depending on $r$ as $n\to\infty$.  
\end{proof}

\subsection{Strong homogenization} Finally, we prove the strong homogenization theorem, Theorem \ref{thm1}. 

\begin{proof}[Proof of Theorem \ref{thm1}]
By Theorem \ref{thm3}, the forms $C^{-1}3^{n(d_w-d_f)}\mcE_{F_n}$ are Mosco convergent to $\mcE_B$, where $C$ is the constant from Theorem \ref{thm3}. We abbreviate  $\bar{R}_{F_n}$ for the effective resistance associated with $C^{-1}3^{n(d_w-d_f)}\mcE_{F_n}$. Then
\[
H_n(r)=\frac{\bar{R}_{F_n}(L_1,L_3)}{\bar{R}_{F_n}(L_2,L_4)}.
\]
To prove the theorem, it suffices to establish the limits
\begin{align}
\lim_{n\to\infty}\bar{R}_{F_n}(L_1,L_3)=R_B(L_1,L_3),\label{e:6.1}\\  \lim_{n\to\infty}\bar{R}_{F_n}(L_2,L_4)=R_B(L_2,L_4),\label{e:6.2}
\end{align}
where $R_B$ is the effective resistance associated with $\mcE_B$.
We prove only \eqref{e:6.1} here; \eqref{e:6.2} follows by the same argument.

\begin{enumerate}
\item Take $f_n\in W^{1,2}(F_n)$ satisfying 
\[\qquad
f_n|_{L_1}=0,\ f_n|_{L_3}=1\ \hbox{ and }\ 
C^{-1}3^{n(d_w-d_f)}\mcE_{F_n}(f_n)=\frac{1}{\bar{R}_{F_n}(L_1,L_3)}. 
\]
As in the proof of \cite[Lemma 5.4]{CC}, for any subsequence there exists a further subsequence $n_l$, $l\geq 0$ and a function $f\in C(F)$ such that
\[\qquad
f_{n_l}(z_{n_l})\to f(z)\quad\hbox{ for any }z_{n_l}\in F_{n_l}\hbox{ and }z\in F\hbox{ with } z_{n_l}\to z.
\]
Consequently,
\begin{equation*}
f|_{L_1}=0,\ f|_{L_3}=1.  
\end{equation*}
Moreover, by \cite[Lemma 2.11-(b)]{CC}, the sequence $f_{n_l}$ converges strongly to $f$ in $L^2$. Hence,  by Mosco convergence (Theorem \ref{thm3}),
\[\qquad\quad 
\liminf_{l\to\infty}\frac{1}{\bar{R}_{F_{n_l}}(L_1,L_3)}=\liminf_{l\to\infty}C^{-1}3^{n_l(d_w-d_f)}\mcE_{F_{n_l}}(f_{n_l})\geq \mcE_B(f)\geq \frac{1}{R_B(L_1,L_3)}. 
\]
Since this holds for every subsequence, we obtain 
\begin{equation}\label{e:6.3} 
\liminf_{n\to\infty}\frac{1}{\bar{R}_{F_n}(L_1,L_3)}\geq \frac{1}{R_B(L_1,L_3)}.
\end{equation} 

\item Choose $h\in\mcF$ such that 
\[
h|_{L_1}=0,\ h|_{L_3}=1\ \hbox{ and }\ \mcE_B(h)=1/R_B(L_1,L_3).
\]
By \cite[Lemma 5.5]{CC}, there exists $h_n\in W^{1,2}(F_n)$ such that $h_n|_{F}$ converges  uniformly to $h$ and 
\[
\lim_{n\to\infty}C^{-1}3^{n(d_w-d_f)}\mcE_{F_n}(h_n)=\mcE_B(h)=\frac{1}{R_B(L_1,L_3)}.
\]
For any small $\delta>0$, we have $h_n|_{L_1}<\delta$ and $h_n|_{L_3}>1-\delta$ for all sufficiently large $n$. Therefore, 
\[
C^{-1}3^{n(d_w-d_f)}\mcE_{F_n}(h_n)\geq \frac{(1-2\delta)^2}{\bar R_{F_n}(L_1,L_3)}.
\]
Letting $\delta\to 0$ gives
\begin{equation} \label{e:6.4} 
\limsup_{n\to\infty}\frac{1}{\bar{R}_{F_n}(L_1,L_3)}\leq \frac{1}{R_B(L_1,L_3)}. 
\end{equation} 
\end{enumerate}
Combining \eqref{e:6.3} and \eqref{e:6.4} yields \eqref{e:6.1}. This completes the proof. 
\end{proof}

\bibliographystyle{amsplain}

\begin{thebibliography}{10}
\bibitem{AB}
S. Andres and M.T. Barlow, \emph{Energy inequalities for cutoff functions and some applications}, J. Reine Angew. Math. 699 (2015), 183--215.

\bibitem{BB}
M.T. Barlow and R.F. Bass, \emph{The construction of Brownian motion on the Sierpinski carpet}, Ann. Inst.
Henri Poincar\'{e} 25 (1989), no. 3, 225--257.

\bibitem{BB1} M.T. Barlow and R.F. Bass,
\emph{Local times for Brownian motion on the Sierpiński carpet,} Probab. Theory Related Fields 85 (1990), no. 1, 91--104.

\bibitem{BB4}
M.T. Barlow and R.F. Bass, \emph{On the resistance of the Sierpiński carpet,} Proc. Roy. Soc. London Ser. A  431 (1990), no. 1882, 345--360.

\bibitem{BB2}
M.T. Barlow and R.F. Bass, \emph{Transition densities for Brownian motion on the Sierpinski carpet,} Probab. Theory Related Fields 91 (1992), 307--330.


\bibitem{BB3}
M.T. Barlow and R.F. Bass, \emph{Brownian motion and harmonic analysis on Sierpinski carpets,} Canad. J. Math. 51 (1999), no. 4, 673--744.

\bibitem{BBKT} M.T. Barlow, R.F. Bass, T. Kumagai and A. Teplyaev, \emph{Uniqueness of Brownian motion on Sierpinski carpets}, J. Eur. Math. Soc. 12 (2010), no. 3, 655--701.

\bibitem{BCM} 
M.T. Barlow, Z.-Q. Chen and M. Murugan, \emph{Stability of EHI and regularity of MMD spaces}, arXiv: 2008.05152.


\bibitem{BHHW} M.T. Barlow, K. Hattori, T. Hattori and H. Watanabe, \emph{Weak homogenization of anisotropic diffusion on pre-Sierpiński carpets.} Comm. Math. Phys. 188 (1997), no. 1, 1--27.


\bibitem{CC} S. Cao and Z.-Q. Chen, \emph{Convergence of resistances on generalized Sierpiński carpets}, J. Eur. Math. Soc. (2026), to appear. DOI: 10.4171/JEMS/1775.


\bibitem{CQ2} S. Cao and H. Qiu,
\emph{Uniqueness and convergence of resistance forms on unconstrained Sierpinski carpets}, Trans. Amer. Math. Soc. (2026), to appear. DOI: 10.1090/tran/9729.

\bibitem{CF} Z.-Q. Chen and M. Fukushima, \emph{Symmetric Markov processes, time change, and boundary theory}. London Math. Soc. Monogr. Ser., 35 Princeton University Press, Princeton, NJ, 2012, xvi+479 pp.


\bibitem{FOT} M. Fukushima, Y. Oshima and M. Takeda, \emph{Dirichlet forms and symmetric Markov processes.} Second revised and extended edition,  De Gruyter Studies in Mathematics, 19. Walter de Gruyter \& Co., Berlin, 2011.


\bibitem{GHL} A. Grigor\'yan, J. Hu and K.-S. Lau, \emph{Heat kernels on metric measure spaces and an application to semilinear elliptic equations}, Trans. Amer. Math. Soc. 355 (2003), no. 5, 2065--2095.


\bibitem{GHL2}
A. Grigor\'yan, J. Hu and K.-S. Lau, \emph{Generalized capacity, Harnack inequality and heat kernels of Dirichlet forms on metric measure spaces}, J. Math. Soc. Japan 67 (2015), no. 4, 1485--1549.

\bibitem{Hino}
M. Hino, \emph{Upper estimate of martingale dimension for self-similar fractals},  Probab. Theory Related Fields 156 (2013), no. 3-4, 739--793.

\bibitem{HZ}
J. Hu and M. Z\"ahle, \emph{Potential spaces on fractals}, Studia Math. 170 (2005), no. 3, 259--281.


\bibitem{Ka}
N. Kajino, \emph{An elementary proof that walk dimension is greater than two for Brownian motion on Sierpi\'nski carpets}, 
Bull. Lond. Math. Soc. 55 (2023), no. 1, 508--521.



\bibitem{KZ}
S. Kusuoka and X.Y. Zhou, \emph{Dirichlet forms on fractals: Poincar\'{e} constant and resistance}, Probab. Theory Related Fields 93 (1992), no. 2, 169--196.

\bibitem{Mosco}
U.~Mosco, \emph{Composite media and asymptotic {D}irichlet forms}, J. Funct. Anal. 123 (1994), no.~2, 368--421.

\bibitem{Murugan}
M. Murugan, \emph{Heat kernel for reflected diffusion and extension property on uniform domains}, Probab. Theory Related Fields 190 (2024), no. 1-2, 543--599.
\end{thebibliography}

\end{document}